\documentclass[final,a4paper]{amsart}
\usepackage{amsmath,amssymb,amsfonts}
\usepackage[foot]{amsaddr}
\usepackage{color}
\usepackage{verbatim}
\usepackage{eepic,epic}
\usepackage{epsfig,subfigure,epstopdf}
\usepackage[]{graphics}
\usepackage{graphicx,inputenc}
\usepackage{subfigure}
\usepackage[notref,notcite]{showkeys}

\usepackage[colorlinks,linktocpage,linkcolor=blue]{hyperref}
\usepackage{fancyhdr}

\newcommand{\red}[1]{\textcolor{red}{#1}}

\newtheorem{theorem}{Theorem}[section]
\newtheorem{lemma}{Lemma}[section]
\newtheorem{definition}[theorem]{Definition}

\newtheorem{assumption}[theorem]{Assumption}
\theoremstyle{remark}
\newtheorem{remark}{Remark}[section]
\newtheorem{corollary}{Corollary}[section]
\numberwithin{equation}{section}

\def\D{^C\!D_t^{\alpha}}
\def\N+{n\in\mathbb{N}^{+}}
\def\NN{n\in\mathbb{N}}
\def\n{\partial {\overrightarrow{\bf n}}}
\def\DK{\mathcal{D}(K)}
\def\DKd{\mathcal{D}(K_\delta)}
\def\I{I_t^\alpha}
\def\L{\mathcal{L}}
\begin{document}

\title
{An undetermined time-dependent coefficient in a fractional diffusion equation}
\author {Zhidong Zhang}
\address{Department of Mathematics, Texas A\&M University, College Station, TX, 77843, USA 
}
\email{zhidong@math.tamu.edu}
\begin{abstract}
In this work, we consider a FDE (fractional diffusion equation) 
$$\D u(x,t)-a(t)\L u(x,t)=F(x,t)$$ with a time-dependent diffusion 
coefficient $a(t)$. This is an extension of \cite{Zhang2016undetermined}, 
which deals with this FDE in one-dimensional space.  
For the direct problem, given an $a(t),$ 
we establish the existence, uniqueness and some regularity properties 
with a more general domain $\Omega$ and right-hand side $F(x,t)$. 
For the inverse problem--recovering $a(t),$ 
we introduce an operator $K$ one of whose fixed points is $a(t)$ 
and show its monotonicity, uniqueness and 
existence of its fixed points. With these properties, a reconstruction 
algorithm for $a(t)$ is created and some numerical results are provided 
to illustrate the theories. \\\\
\text{Keywords}: fractional diffusion, fractional inverse problem,  
uniqueness, existence, monotonicity, 
iteration algorithm. 
\end{abstract}

\maketitle

\section{Introduction}

\par This paper considers the fractional diffusion equation (FDE) with 
a continuous and positive coefficient function $a(t):$
\begin{equation}\label{fde}
  \begin{aligned}
   \D u(x,t)-a(t)\L u(x,t)&=F(x,t),\ &&x\in \Omega,\ t\in (0,T];\\
   u(x,t)&=0,\ &&(x,t)\in \partial\Omega \times (0,T];\\
   u(x,0)&=u_0(x),\ &&x\in \Omega,
   \end{aligned}
 \end{equation}
 where $\Omega$ is a bounded and smooth subset of $R^n, n=1,2,3,$ $-\L$ is a symmetric uniformly elliptic operator defined as 
 $$-\L u=-\sum_{i,j=1}^n (a^{ij}(x)u_{x_i})_{x_j}+c(x)u$$ with conditions 
 \begin{equation}\label{sobolev assumption}
 a^{ij},c\in C^2(\overline{\Omega}) \ (i,j=1,\dots,n),\ \partial \Omega\ \text{is}\  C^3,
 \end{equation} 
 and $\D$ is the left-sided Djrbashian-–Caputo $\alpha$-th order derivative with respect to time $t.$ The definition for $\D$ is $$\D u(x,t)=\frac{1}{\Gamma(n-\alpha)}
 \int_{0}^{t} (t-\tau)^{n-\alpha-1}
 \frac{d^n}{d\tau^n}u(x,\tau) {\rm d}\tau$$
 with Gamma function $\Gamma(\cdot)$ and the nearest integer $n$ with 
 $\alpha\le n.$ In this paper, we are assuming a subdiffusion process,  
 i.e. $\alpha\in(0,1).$ This simplifies the definition of $\D$ as 
 $$ \D u(x,t)=\frac{1}{\Gamma(1-\alpha)}
 \int_{0}^{t} (t-\tau)^{-\alpha}\frac{d}{d\tau}u(x,\tau) {\rm d}\tau.$$
 This work is an extension of \cite{Zhang2016undetermined} 
 from a simple space domain $\Omega$ to $\mathbb{R}^n$, considers the 
 more general analysis for the direct problem and contains an existence 
 argument for the inverse problem of recovering $a(t).$

\par This paper consists of two parts; the direct problem and 
the inverse problem. For the direct problem, we build the 
spectral representation of the weak solution $u(x,t;a).$ The notation $u(x,t;a)$ is used for displaying the dependence of the solution $u$ on the diffusivity $a(t).$ 
Then the existence, uniqueness and regularity results are proved 
with several assumptions on the coefficient function $a(t).$ 
Unlike \cite{Zhang2016undetermined}, 
the right hand side function $F(x,t)$ is not of the form $f(x)g(t)$, so 
that the proof of regularity is more delicate. For the inverse problem, we use the single point flux data 
$$a(t)\frac{\partial u}{\n}(x_0,t;a)=g(t),\ x_0\in \partial \Omega$$
to recover the coefficient $a(t)$ 
(We choose the data $a(t)\frac{\partial u}{\n}(x_0,t;a)=g(t)$ instead of 
the classical flux $\frac{\partial u}{\n}(x_0,t;a)$ because 
in practice, $a(t)\frac{\partial u}{\n}(x_0,t;a)$ is usually  
measured as the flux). 
For the reconstruction, we only consider to recover a 
continuous and positive $a(t)$ to match the assumptions set in the direct 
problem. Acting a flux data, we introduce an operator $K$ 
one of whose fixed points is the coefficient $a(t).$ Using 
the weak maximum principle \cite{luchko2009maximum}, we establish the monotonicity and uniqueness of the fixed points of operator $K$, and the proof of uniqueness leads to a numerical reconstruction algorithm. Since we consider a multidimensional domain $\Omega$ here, the Sobolev Embedding Theorem yields that we need to add the condition \eqref{sobolev assumption} on the operator $-\L$ to ensure the $C^1$-regularity of the series representation of $u$. Then the operator $K$ is well-defined, where the proofs can be seen in section 4. This is a significant difference from  \cite{Zhang2016undetermined}.  
Furthermore, an existence argument of the fixed points of $K$ is included by this paper, which \cite{Zhang2016undetermined} does not contain. 

\par The rest of this paper follows the following structure. In section 2, we collect 
some preliminary results about fractional calculus and the eigensystem 
of $-\L$. The direct problem 
is discussed in section 3, i.e. we establish the existence, uniqueness 
and some regularity results of the weak solution for FDE \eqref{fde}. 
Then section 4 deals with the inverse problem of recovering $a(t)$. Specifically, an operator 
$K$ is introduced at the beginning of this section, then its monotonicity 
and uniqueness of its fixed points give an algorithm to recover the 
coefficient $a(t)$. In particular, the existence argument of the fixed 
points of $K$ is included by this section. In section 5, some 
numerical results are presented to illustrate the theoretical basis.

\section{Preliminary material}

\subsection{Mittag-Leffler function}
\par In this part, we describe the Mittag-Leffler function which plays 
an important role in fractional diffusion equations. This  
is a two-parameter function defined as 
\begin{equation*}
  E_{\alpha,\beta}(z) = \sum_{k=0}^\infty \frac{z^k}{\Gamma(k\alpha+\beta)},
  \ z\in \mathbb{C}.
\end{equation*}
It generalizes the natural exponential function in the sense that 
$E_{1,1}(z)=e^z$. We list some important properties of the 
Mittag-Leffler function for future use.

\begin{lemma}\label{mittag_bound}
Let $0<\alpha<2$ and $\beta\in\mathbb{R}$ be arbitrary, and $\frac{\alpha\pi}{2}
<\mu<\min(\pi,\alpha\pi)$. Then there exists a constant $C=C(\alpha,\beta,\mu)>0$ such
that
\begin{equation*}
  |E_{\alpha,\beta}(z)|\leq \frac{C}{1+|z|},\quad \mu\leq|\mathrm{arg}(z)|\leq \pi.
\end{equation*}
\end{lemma}
\begin{proof}
 This proof can be found in \cite{kilbas2006theory}.
\end{proof}

\begin{lemma}\label{mittag_derivative}
 \par For $\lambda>0,\ \alpha>0$ and $\N+,$ we have 
 $$\frac{d^n}{dt^n}E_{\alpha,1}(-\lambda t^\alpha)
 =-\lambda t^{\alpha-n}E_{\alpha,\alpha-n+1}(-\lambda t^\alpha),
 \ t>0.$$
In particular, if we set $n=1,$ 
 then there holds 
 $$\frac{d}{dt}E_{\alpha,1}(-\lambda t^\alpha)
 =-\lambda t^{\alpha-1}E_{\alpha,\alpha}(-\lambda t^\alpha),
 \ t>0.$$
\end{lemma}

\begin{proof}
 \par This is \cite[Lemma $3.2$]{sakamoto2011initial}.
\end{proof}

\begin{lemma}\label{mittag_positive}
 If $0<\alpha<1$ and $z>0,$ then 
 $E_{\alpha,\alpha}(-z)\ge 0.$
\end{lemma}
\begin{proof}
 This proof can be found in \cite{miller2001completely,pollard1948completely,schneider1996completely}.
\end{proof}

\begin{lemma}\label{mittag_positive_1}
\par For $0<\alpha<1,$ $E_{\alpha,1}(-t^\alpha)$ is completely monotonic, 
that is,  
$$(-1)^n \frac{d^n}{dt^n}E_{\alpha,1}(-t^\alpha)\ge 0,\ for\ t>0
\ and\ n=0,1,2,\cdots.$$
\end{lemma}

\begin{proof}
 See \cite{gorenflofractional}.
\end{proof}

\subsection{Fractional calculus} 

\par In this part, we collect some results of fractional calculus. 
The next lemma states the extremal principle of ${\D}.$
\begin{lemma}\label{extreme}
 Fix $0<\alpha<1$ and given $f(t)\in C[0,T]$ with $\D f \in C[0,T].$ 
 If $f$ attains its maximum (minimum) over the interval $[0,T]$ at the point 
 $t=t_0,\ t_0\in (0,T],$ then 
 $^C\!D_{t_0}^{\alpha} f \ge(\le) 0.$ 
\end{lemma}
\begin{proof}
 \par Even though the conditions are different from the ones of \cite[Theorem 1]{luchko2009maximum}, the maximum case can be proved following the proof of  \cite[Theorem 1]{luchko2009maximum}. 
 For the minimum case, we only need to set $\overline{f}=-f.$ 
\end{proof}

\par The following lemma about the composition between 
$\D$ and the fractional integral $\I$ is presented in 
\cite{samko1993fractional}.

\begin{lemma}\label{I_alpha}
Define the Riemann-–Liouville $\alpha$-th order integral $I_t^\alpha$ as 
$$
 \I u=\frac{1}{\Gamma(\alpha)}\int_{0}^{t} (t-\tau)^{\alpha-1}
 u(\tau) {\rm d}\tau.
 $$
 For $0<\alpha<1,$ $u(t),{\D} u\in C[0,T],$ we have
$$(\D \circ \I u)(t)=u(t),\quad (\I\circ {\D} u)(t)=u(t)-u(0),\ t\in[0,T].$$
\end{lemma}

\subsection{Eigensystem of $-\L$}
\par Since $-\L$ is a symmetric uniformly elliptic operator, 
we denote the eigensystem of $-\L$ by $\{(\lambda_n,\phi_n):\N+\}.$ 
Then we have $0<\lambda_1\le \lambda_2\le \cdots$ where finite multiplicity is possible, $\lambda_n\to \infty$ and $\{\phi_n:\N+\}\subset H^2(\Omega)\cap 
H_0^1(\Omega)$ forms an orthonormal basis of $L^2(\Omega).$ 
\par Moreover, with the condition \eqref{sobolev assumption}, for each $\N+$, it holds that $\phi_n \in H^3(\Omega)$ \cite{Evans2010}. 
Then by the Sobolev Embedding Theorem, we have $\phi_n\in C^1(\overline{\Omega})$ and  
$\frac{\partial \phi_n}{\n}(x_0)$ is well-defined for each $\N+$. Hence,  without loss of generality, we can suppose 
\begin{equation}\label{sign_eigenfunction_derivative}
\frac{\partial \phi_n}{\n}(x_0)\ge 0,\ \text{for each}\ \N+.
\end{equation}
Otherwise, if $\frac{\partial \phi_k}{\n}(x_0)< 0$ for some 
$k\in \mathbb{N}^{+},$ we can replace $\phi_k$ by $-\phi_k.$ 
$-\phi_k$ satisfies all the properties we need, such as it is 
an eigenfunction of $-\L$ corresponding to the eigenvalue $\lambda_k,$ 
composes an orthonormal basis of $L^2(\Omega)$ together with 
$\{\phi_n:\N+,n\ne k\}$ and $\frac{\partial (-\phi_k)}{\n}(x_0)\ge 0.$
The assumption \eqref{sign_eigenfunction_derivative} will be used in Section 4.

\section{Direct Problem--Existence, Uniqueness and Regularity}
\par Throughout this section, we suppose $a(t),$ $u_0(x)$ and $F(x,t)$ 
satisfy the following assumptions:
\begin{assumption}\label{assumption_direct}
~
\begin{itemize}
 \item[(a)] $a(t)\in C^{+}[0,T]:=\{\psi\in C[0,T]:\psi(t)>0,\ t\in[0,T]\};$
 \item[(b)] $F(x,t) \in C([0,T];L^2(\Omega));$
 \item[(c)] $u_0(x) \in H_0^1(\Omega).$ 
\end{itemize}
\end{assumption}

\subsection{Spectral Representation}

\begin{definition}\label{weak solution}
   We call $u(x,t;a)$ a weak solution of  
   FDE \eqref{fde} in $L^2(\Omega)$ corresponding to the  
   coefficient $a(t)$ if 
   $u(\cdot,t;a)\in H_0^1(\Omega)$ for $t\in(0,T]$ and 
   for any $\psi(x)\in H^2(\Omega)\cap H_0^1(\Omega),$ it holds  
   \begin{equation*}
   \begin{split}
    &(\D u(x,t;a),\psi(x))-(a(t)\L u(x,t;a),\psi(x))=(F(x,t),\psi(x)),
   \ t\in(0,T];\\
   &(u(x,0;a),\psi(x))=(u_0(x),\psi(x)),
   \end{split}
   \end{equation*}
where $(\cdot,\cdot)$ is the inner product in $L^2(\Omega).$
\end{definition}

\par With the above definition, we give a spectral representation 
for the weak solution in the following lemma. 

\begin{lemma}\label{spectral representation}
 Define $b_n:=(u_0(x),\phi_n(x)), F_n(t)=(F(x,t),\phi_n(x)),
 \ \N+.$
 The spectral representation of the weak solution of FDE  
 \eqref{fde} is 
 \begin{equation}\label{solution}
  u(x,t;a)=\sum_{n=1}^\infty u_n(t;a)\phi_n(x),\ (x,t)\in \Omega\times [0,T], 
 \end{equation}
 where 
 $u_n(t;a)$ satisfies the fractional ODE
 \begin{equation}\label{ODE}
 \D u_n(t;a)+\lambda_n a(t) u_n(t;a)=F_n(t),
  \ u_n(0;a)=b_n,\ \N+.
 \end{equation}

\end{lemma}

\begin{proof}
 \par For each $\N+,$ multiplying $\phi_n(x)$ on both sides of 
 FDE \eqref{fde} and integrating it on $x$ over $\Omega$ 
 allow us to deduce that   
\begin{equation}\label{equality_1}
 \D( u(x,t;a), \phi_n(x))+\lambda_n a(t)(u(x,t;a),\phi_n(x))=F_n(t),
\end{equation}
where 
$(-\L u(x,t;a),\phi_n(x))=(u(x,t;a),-\L\phi_n(x))
=\lambda_n (u(x,t;a),\phi_n(x))$ follows from the symmetricity of $-\L.$
Set $u_n(t;a)=(u(x,t;a), \phi_n(x))$ and define 
$u(x,t;a)=\sum\limits_{n=1}^\infty u_n(t;a)\phi_n(x).$ 
Then \eqref{equality_1} and the completeness of 
$\{\phi_n(x):\N+\}$ lead to the desired result.
\end{proof}

\subsection{Existence and Uniqueness}
\par In order to show the existence and uniqueness 
of the weak solution \eqref{solution}, we state the following lemma 
\cite[Theorem 3.25]{kilbas2006theory}.

\begin{lemma}\label{kilbassrivastavatrujillo}
 For the Cauchy-type problem
 $$\D y=f(y,t),\ y(0)=c_0,$$
 if for any continuous $y(t),$ $f(y,t)\in C[0,T]$, $\exists A>0$ 
 which is independent of $y\in C[0,T]$ and $t\in[0,T]$ s.t. 
 $\lvert f(t,y_1)-f(t,y_2)\rvert \le A\lvert y_1-y_2\rvert,$
then there exists a unique solution $y(t)$ for the Cauchy-type 
problem, which satisfies $\D y\in C[0,T]$.
\end{lemma}

\par The theorem of existence and uniqueness for $u(x,t;a)$ follows from 
Lemma \ref{kilbassrivastavatrujillo}.
\begin{theorem}[Existence and Uniqueness]\label{existenceuniqueness}
 Suppose Assumption \ref{assumption_direct} holds. Under 
 Definition \ref{weak solution}, there exists a unique weak solution 
 $u(x,t;a)$ of FDE \eqref{fde} with the spectral 
 representation \eqref{solution} and
 for each $\N+,$ $u_n(t;a)\in C[0,T]$ is the unique solution of 
 the fractional ODE \eqref{ODE} with $\D u_n(t;a) \in C[0,T].$ 
\end{theorem}

\begin{proof}
\par From the spectral representation \eqref{solution}, it 
suffices to show the existence and uniqueness of $u_n(t;a),\N+.$
Fix $\N+,$ Assumption \ref{assumption_direct} $(a)$ and $(b)$ yield 
that the fractional ODE \eqref{ODE} satisfies the conditions of 
Lemma \ref{kilbassrivastavatrujillo}. Hence the existence and uniqueness 
for $u_n(t;a)$ hold.
\end{proof}

\subsection{Sign of $u_n(t;a)$}
\par In this part, we state two properties of $u_n(t;a)$ which play important 
roles in building the regularity of $u(x,t;a).$
\begin{lemma}\label{sign}
 Given $h\in C^{+}[0,T],$ $f\in C[0,T]$ with $\D f \in C[0,T],$ 
 if $f(0)\le (\ge)0$ and $\D f+h(t)f(t)\le (\ge)0,$ then $f\le 
 (\ge)0$ on $[0,T]$.
\end{lemma}
\begin{proof}
 \par Since $f(t)\in C[0,T]$, $f(t)$ attains its maximum over $[0,T]$ 
 at some point $t_0\in [0,T]$. 
 If $t_0=0,$ then $f(t)\le f(0)\le 0.$
 If $t_0\in (0,T],$ with Lemma \ref{extreme}, we have 
 ${}^C\!D_{t}^{\alpha} f(t_0)\ge 0$, which yields $h(t_0)f(t_0)\le 0,$ i.e. 
 $f(t_0)\le 0$ due to $h>0$ on $[0,T].$ 
 The definition of $t_0$ assures $f\le 0$.
 \par For the case of ``$\ge 0$", let $\overline{f}(t)=-f(t),$ then the 
 above proof gives $\overline{f}\le 0,$ i.e. $f\ge 0$.
\end{proof}

\par The following corollary, which concerns the sign of $u_n(t;a)$, follows 
from Lemma \ref{sign} directly. 
\begin{corollary}\label{sign_eigenfunction}
 Set $u_n(t;a)$ be the unique solution of 
 the fractional ODE \eqref{ODE}. Then 
 $\D u_n(t;a)+\lambda_n a(t) u_n(t;a)\le (\ge) 0$ on $[0,T]$ and 
 $u_n(0;a) \le (\ge) 0$ imply  $u_n(t;a)\le (\ge)0$ on $[0,T],\ \N+$. 
\end{corollary}
\begin{proof}
 Assumption \ref{assumption_direct} gives that $\lambda_na(t)\in 
 C^{+}[0,T].$ Then the proof is completed by applying Lemma \ref{sign} 
 to the fractional ODE \eqref{ODE}. 
\end{proof}

\subsection{Regularity}

\par In this part, we establish the regularity of $u(x,t;a)$. 
To this end, we split FDE \eqref{fde} into 
 \begin{equation}\label{model_right}
  \begin{aligned}
   \D u(x,t)-a(t)\L u(x,t)&=F(x,t),\ &&x\in \Omega,\ t\in (0,T];\\
   u(x,t)&=0,\ &&(x,t)\in \partial\Omega \times (0,T];\\
   u(x,0)&=0,\ &&x\in \Omega,
   \end{aligned}
 \end{equation}
 and
 \begin{equation}\label{model_initial}
  \begin{aligned}
   \D u(x,t)-a(t)\L u(x,t)&=0,\ &&x\in \Omega,\ t\in (0,T];\\
   u(x,t)&=0,\ &&(x,t)\in \partial\Omega \times (0,T];\\
   u(x,0)&=u_0(x),\ &&x\in \Omega. 
   \end{aligned}
 \end{equation}

\par Denote the weak solutions of FDEs \eqref{model_right} and 
\eqref{model_initial} by $u^r(x,t;a)$ and $u^i(x,t;a)$, respectively
(``r" and ``i" denote the initials of ``right-hand side" and ``initial 
condition"). 
The following lemma about $u^r(x,t;a)$ and $u^i(x,t;a)$ 
follows from Lemma \ref{spectral representation} and 
Theorem \ref{existenceuniqueness}.

\begin{lemma}
Suppose Assumption \ref{assumption_direct} holds. Then
 $u^r(x,t;a)$ and $u^i(x,t;a)$ are the unique solutions for FDEs 
 \eqref{model_right} and \eqref{model_initial}, respectively, with the 
 spectral representations as  
\begin{equation}\label{u^r u^i}
 u^r(x,t;a)=\sum_{n=1}^\infty u^r_n(t;a)\phi_n(x),\ 
u^i(x,t;a)=\sum_{n=1}^\infty u^i_n(t;a)\phi_n(x),
\end{equation}
where $u^r_n(t;a),\ u^i_n(t;a)$ satisfy the following fractional ODEs 
\begin{equation}\label{ODE_right}
 \D u_n^r(t;a)+\lambda_na(t)u_n^r(t;a)=F_n(t),\ u_n^r(0;a)=0,\ \N+;
\end{equation}
\begin{equation}\label{ODE_initial}
 \D u_n^i(t;a)+\lambda_na(t)u_n^i(t;a)=0,\ u_n^i(0;a)=b_n,\ \N+.
\end{equation}
\par Moreover, Theorem \ref{existenceuniqueness} ensures the weak 
solution $u(x,t;a)$ of FDE \eqref{fde} 
can be written as $u(x,t;a)=u^r(x,t;a)+u^i(x,t;a),$ i.e. 
$u_n(t;a)=u^r_n(t;a)+u^i_n(t;a),\ \N+.$
\end{lemma}

\subsubsection{Regularity of $u^r$}
\par For each $\N+,$ define 
\begin{equation}\label{+-}
 F^{+}_n(t)=\begin{cases}
  F_n(t), &\text{if}\ F_n(t)\ge0;\\
  0, &\text{if}\ F_n(t)<0,
 \end{cases}
 \quad F^{-}_n(t)=\begin{cases}
  F_n(t),&\text{if}\ F_n(t)<0;\\
  0,&\text{if}\ F_n(t)\ge0.
 \end{cases}
\end{equation}
It is obvious that $F_n=F^{+}_n+F^{-}_n,$ the supports 
of $F^{+}_n$ and $F^{-}_n$ are disjoint and 
$F^{+}_n,F^{-}_n \in C[0,T]$ which follows from $F_n\in C[0,T]$. 
Split $u_n^r(t;a)$ as $u_n^r(t;a)=u_n^{r,+}(t;a)+u_n^{r,-}(t;a),$ 
where $u_n^{r,+}(t;a),\ u_n^{r,-}(t;a)$ satisfy 
\begin{equation}\label{ODE_right_+}
 \D u_n^{r,+}(t;a)+\lambda_na(t)u_n^{r,+}(t;a)=F^{+}_n(t),\ u_n^{r,+}(0;a)=0,\ \N+;
\end{equation}
\begin{equation}\label{ODE_right_-}
 \D u_n^{r,-}(t;a)+\lambda_na(t)u_n^{r,-}(t;a)=F^{-}_n(t),\ u_n^{r,-}(0;a)=0,\ \N+,
\end{equation}
respectively. The existence and uniqueness of 
$u_n^{r,+}(t;a)$ and $u_n^{r,-}(t;a)$ hold due to Lemma 
\ref{kilbassrivastavatrujillo} and we can write 
\begin{equation}\label{u^r u^{r,+} u^{r,-}}
u^r(x,t;a)=u^{r,+}(x,t;a)+u^{r,-}(x,t;a),
\end{equation}
where
\begin{equation}\label{u^{r,+} u^{r,-}}
u^{r,+}(x,t;a)=\sum\limits_{n=1}^{\infty}u_n^{r,+}(t;a)\phi_n(x),
\ u^{r,-}(x,t;a)=\sum\limits_{n=1}^{\infty}u_n^{r,-}(t;a)\phi_n(x).
\end{equation}

\par Then we state some properties of $u_n^{r,+}(t;a)$ 
and $u_n^{r,-}(t;a).$

\begin{lemma}\label{sign_right_+-}
For any $\N+,$
 $u_n^{r,+}(t;a)\ge 0$ and $u_n^{r,-}(t;a)\le 0$ on $[0,T].$ 
\end{lemma}
\begin{proof}
 This proof follows from Corollary \ref{sign_eigenfunction} directly.
\end{proof}

\begin{lemma}\label{monotone_right}
 Given $a_1(t), a_2(t)\in C^{+}[0,T]$ with $a_1(t)\le a_2(t)$ on $[0,T]$, 
 we have 
 $$0\le u_n^{r,+}(t;a_2)\le u_n^{r,+}(t;a_1),
\ u_n^{r,-}(t;a_1)\le u_n^{r,-}(t;a_2)\le 0,
  \ t\in [0,T],\ \N+.$$
\end{lemma}

\begin{proof}
 Pick $\N+,$ $u_n^{r,+}(t;a_1)$ and $u_n^{r,+}(t;a_2)$ satisfy 
 the following system:
 \begin{equation*}
  \begin{cases}
   \D u_n^{r,+}(t;a_1)+\lambda_na_1(t)u_n^{r,+}(t;a_1)=F^{+}_n(t);\\
   \D u_n^{r,+}(t;a_2)+\lambda_na_2(t)u_n^{r,+}(t;a_2)=F^{+}_n(t);\\
   u_n^{r,+}(0;a_1)=u_n^{r,+}(0;a_2)=0,
  \end{cases}
 \end{equation*}
which leads to 
$$\D w+\lambda_na_1(t)w(t)=\lambda_n u_n^{r,+}(t;a_2)(a_2(t)-a_1(t))\ge0,
\ w(0)=0,$$ 
where $w(t)=u_n^{r,+}(t;a_1)-u_n^{r,+}(t;a_2)$ and the 
last inequality follows from Lemma \ref{sign_right_+-} and $a_1\le a_2.$ 
Hence, Corollary \ref{sign_eigenfunction} shows that 
$w(t)\ge 0,$ i.e. $u_n^{r,+}(t;a_2)\le u_n^{r,+}(t;a_1)$ and 
Lemma \ref{sign_right_+-} gives 
$0\le u_n^{r,+}(t;a_2)\le u_n^{r,+}(t;a_1),\ t\in [0,T].$
\par Similarly, we have  
$u_n^{r,-}(t;a_1)\le u_n^{r,-}(t;a_2)\le 0,\ t\in [0,T],$ 
completing the proof.
\end{proof}

\par Assumption \ref{assumption_direct} $(a)$ implies 
there exists constants $q_a, Q_a$ s.t. 
\begin{equation}\label{q_a Q_a}
 0<q_a<a(t)<Q_a\ \text{on}\ [0,T].
\end{equation}
 From Lemma \ref{monotone_right}, we obtain 
\begin{equation}\label{inequality_1}
 |u_n^{r,+}(t;a)|\le |u_n^{r,+}(t;q_a)|,\ 
|u_n^{r,-}(t;a)|\le |u_n^{r,-}(t;q_a)|\ \text{on}\ t\in [0,T],\ \N+,
\end{equation}
where $u_n^{r,+}(t;q_a),u_n^{r,-}(t;q_a)$ are the unique solutions of 
fractional ODEs \eqref{ODE_right_+} and \eqref{ODE_right_-} respectively 
with $a(t)\equiv q_a$ on $[0,T]$.
The next two lemmas concern the regularity of 
$u^{r,+}(x,t;a)$ and $\D u^{r,+}(x,t;a),$ respectively.
\begin{lemma}\label{regularity_u^{r,+}}
$$\|u^{r,+}\|_{L^2(0,T;H^2(\Omega))}
\le C\|F\|_{L^2([0,T]\times\Omega)}.$$
\end{lemma}

\begin{proof}
\par Calculating $\|u^{r,+}(x,t;a)\|_{L^2(0,T;H^2(\Omega))}^2$ directly 
yields
\begin{equation*}
\begin{split}
 \|u^{r,+}(x,t;a)\|_{L^2(0,T;H^2(\Omega))}^2
 &=\int_0^T\|u^{r,+}(x,t;a)\|_{H^2(\Omega))}^2 {\rm d}t
 \le \int_0^T C\|(-\L u^{r,+})(x,t;a)\|_{L^2(\Omega)}^2 {\rm d}t\\
 &=C\int_0^T\|\sum\limits_{n=1}^{\infty}\lambda_nu_n^{r,+}(t;a)\phi_n(x)\|
 _{L^2(\Omega)}^2 {\rm d}t\\
 &=C\int_0^T\sum\limits_{n=1}^{\infty}\lambda_n^2 |u_n^{r,+}(t;a)|^2
 {\rm d}t
 \le C\int_0^T\sum\limits_{n=1}^{\infty}\lambda_n^2 
 |u_n^{r,+}(t;q_a)|^2 {\rm d}t,
 \end{split}
\end{equation*}
where the last inequality is obtained from \eqref{inequality_1}. 
By the Monotone Convergence Theorem, we have 
\begin{equation}\label{inequality_5}
 \quad\|u^{r,+}(x,t;a)\|_{L^2(0,T;H^2(\Omega))}^2
 \le C\int_0^T\sum\limits_{n=1}^{\infty}\lambda_n^2 
 |u_n^{r,+}(t;q_a)|^2 {\rm d}t
 =C \sum\limits_{n=1}^{\infty} \int_0^T
 |\lambda_nu_n^{r,+}(t;q_a)|^2 {\rm d}t.
\end{equation}

\par For each $\N+$, \cite{sakamoto2011initial} gives the explicit  
representation of $u_n^{r,+}(t;q_a)$
\begin{equation*}\label{u_n^{r,+}(t;q_a)}
 u_n^{r,+}(t;q_a)=\int_0^t F^+_n(\tau) (t-\tau)^{\alpha-1} 
 E_{\alpha,\alpha}(-\lambda_n q_a(t-\tau)^\alpha){\rm d}\tau,
\end{equation*}
which together with Young's inequality leads to
\begin{equation*}\label{inequality_2}
\begin{split}
 \int_0^T|\lambda_nu_n^{r,+}(t;q_a)|^2 {\rm d}t
 &=\|F_n^{+}(t)*(\lambda_nt^{\alpha-1}
 E_{\alpha,\alpha}(-\lambda_nq_at^\alpha))\|_{L^2[0,T]}^2\\
 &\le \|F_n^{+}\|_{L^2[0,T]}^2 \|\lambda_nt^{\alpha-1}
 E_{\alpha,\alpha}(-\lambda_nq_at^\alpha)\|_{L^1[0,T]}^2.
\end{split} 
\end{equation*}
Lemmas \ref{mittag_derivative}, \ref{mittag_positive} and 
\ref{mittag_positive_1} give the bound of $\|\lambda_nt^{\alpha-1}
 E_{\alpha,\alpha}(-\lambda_nq_at^\alpha)\|_{L^1[0,T]}$
\begin{equation*}
\begin{split}
\|\lambda_nt^{\alpha-1}
 E_{\alpha,\alpha}(-\lambda_nq_at^\alpha)\|_{L^1[0,T]}
&=\int_0^T \big|\lambda_n\tau^{\alpha-1} 
 E_{\alpha,\alpha}(-\lambda_n q_a\tau^\alpha)\big|{\rm d}\tau\\
&=\int_0^T \lambda_n\tau^{\alpha-1} 
 E_{\alpha,\alpha}(-\lambda_n q_a\tau^\alpha){\rm d}\tau\\
&=-q_a^{-1}\int_0^T \frac{d}{d\tau} 
E_{\alpha,1}(-\lambda_nq_a\tau^\alpha){\rm d}\tau\\
&=q_a^{-1}(1-E_{\alpha,1}(-\lambda_nq_aT^\alpha))
\le q_a^{-1};
 \end{split} 
\end{equation*}
while the definition \eqref{+-} provides the bound of 
$\|F_n^{+}\|_{L^2[0,T]}$ as 
$\|F_n^{+}\|_{L^2[0,T]}\le \|F_n\|_{L^2[0,T]}.$
Consequently, it holds $\int_0^T|\lambda_nu_n^{r,+}(t;q_a)|^2 {\rm d}t\le q_a^{-2} 
 \|F_n\|_{L^2[0,T]}^2,\ \N+,$ 
 i.e.  
$$\sum_{n=1}^\infty\int_0^T|\lambda_nu_n^{r,+}(t;q_a)|^2 {\rm d}t
\le q_a^{-2} \sum_{n=1}^\infty\|F_n\|_{L^2[0,T]}^2,$$
which together with \eqref{inequality_5} and the completeness 
of $\{\phi_n(x):\N+\}$ in $L^2(\Omega)$ gives  
\begin{equation*}\label{inequality_6}
\begin{split}
 \|u^{r,+}(x,t;a)\|_{L^2(0,T;H^2(\Omega))}^2
 &\le C\sum\limits_{n=1}^{\infty} \int_0^T
 |\lambda_nu_n^{r,+}(t;q_a)|^2 {\rm d}t\\
 &\le C\sum_{n=1}^\infty\|F_n\|_{L^2[0,T]}^2
 =C\|F\|_{L^2([0,T]\times\Omega)}^2,
 \end{split}
\end{equation*}
where the constant $C$ only depends on $a(t).$ This completes the proof.
\end{proof}

\begin{lemma}\label{regularity_D_u^{r,+}}
 $$
\|\D u^{r,+}\|_{L^2([0,T]\times\Omega)}
\le C\|F\|_{L^2([0,T]\times\Omega)}.
$$
\end{lemma}
\begin{proof}
\par \eqref{ODE_right_+}, \eqref{u^{r,+} u^{r,-}}, 
definition \eqref{+-} and the Monotone Convergence Theorem give  
\begin{equation}\label{inequality_17}
 \begin{split}
 \|\D u^{r,+}\|_{L^2([0,T]\times\Omega)}^2
  &=\int_0^T \|\sum_{n=1}^{\infty} {\D} u_n^{r,+}(\cdot;a) \phi_n(x)\|
  _{L^2(\Omega)}^2{\rm d}t
   =\sum_{n=1}^{\infty} \int_0^T |\D u_n^{r,+}(\cdot;a)|^2 {\rm d}t\\
 & \le \sum_{n=1}^{\infty} \int_0^T \left(2|\lambda_n a(t) u_n^{r,+}(t;a)|^2
 +2|F_n^{+}(t)|^2\right){\rm d}t\\
 &\le2\sum_{n=1}^{\infty} \int_0^T |\lambda_n a(t) u_n^{r,+}(t;a)|^2
 {\rm d}t +2\sum_{n=1}^{\infty} \int_0^T |F_n(t)|^2{\rm d}t.
 \end{split}
\end{equation}
The estimate of $\sum\limits_{n=1}^{\infty} \int_0^T 
|\lambda_n a(t) u_n^{r,+}(t;a)|^2{\rm d}t$ follows from 
\eqref{q_a Q_a}, \eqref{inequality_1} and the proof of Lemma 
\ref{regularity_u^{r,+}}
\begin{equation*}\label{inequality_8}
 \sum_{n=1}^{\infty} \int_0^T|\lambda_n a(t) u_n^{r,+}(t;a)|^2{\rm d}t
 \le Q_a \sum_{n=1}^{\infty} \int_0^T
 |\lambda_n u_n^{r,+}(t;q_a)|^2{\rm d}t
 \le C\|F\|_{L^2([0,T]\times\Omega)}^2;
\end{equation*}
while the completeness of $\{\phi_n(x):\N+\}$ gives 
$ \sum\limits_{n=1}^{\infty} \int_0^T |F_n(t)|^2{\rm d}t
 =\|F\|_{L^2([0,T]\times\Omega)}^2.$
Hence, \eqref{inequality_17} develops 
$\|\D u^{r,+}\|_{L^2([0,T]\times\Omega)}^2
\le C\|F\|_{L^2([0,T]\times\Omega)}^2,$ 
which implies the indicated conclusion.
\end{proof}

\par The following corollary follows immediately from the proofs of 
Lemmas \ref{regularity_u^{r,+}} and \ref{regularity_D_u^{r,+}}.  

\begin{corollary}\label{regularity_u^{r,-}}
 \begin{equation*}
\|u^{r,-}\|_{L^2(0,T;H^2(\Omega))}
\le C\|F\|_{L^2([0,T]\times\Omega)},
 \  \|\D u^{r,-}\|_{L^2([0,T]\times\Omega)}
\le C\|F\|_{L^2([0,T]\times\Omega)}.
 \end{equation*}
\end{corollary}

\par From Lemmas \ref{regularity_u^{r,+}}, \ref{regularity_D_u^{r,+}}, 
Corollary \ref{regularity_u^{r,-}} and \eqref{u^r u^{r,+} u^{r,-}}, 
we are able to deduce the regularity for $u^r(x,t;a)$ and 
$\D u^r(x,t;a)$.

\begin{lemma}[Regularity of $u^r$]\label{regularity_u^r}
 $$\|u^r\|_{L^2(0,T;H^2(\Omega))}
  +\|\D u^r\|_{L^2([0,T]\times\Omega)}
  \le C\|F\|_{L^2([0,T]\times\Omega)}.$$
\end{lemma}
\begin{proof}
 \eqref{u^r u^{r,+} u^{r,-}} gives 
 $u^r(x,t;a)=u^{r,+}(x,t;a)+u^{r,-}(x,t;a),$ which leads to
\begin{equation*}
 \begin{split}
  &\|u^r\|_{L^2(0,T;H^2(\Omega))}
  +\|\D u^r\|_{L^2([0,T]\times\Omega)}\\
  \le&\   \|u^{r,+}\|_{L^2(0,T;H^2(\Omega))}
 +\|u^{r,-}\|_{L^2(0,T;H^2(\Omega))}\\
 &+\|\D u^{r,+}\|_{L^2([0,T]\times\Omega)}
 +\|\D u^{r,-}\|_{L^2([0,T]\times\Omega)}\\
 \le &\ C\|F\|_{L^2([0,T]\times\Omega)}.
 \end{split}
\end{equation*}
\end{proof}

\par If we impose a higher regularity on $F,$ we can obtain the regularity estimate of $\|u^{r}\|_{C([0,T];H^2(\Omega))}$. 
\begin{corollary}\label{Cregularity_u^r}
Under Assumption \ref{assumption_direct}, if $F\in C^\theta([0,T];L^2(\Omega)),\ 0<\theta<1,$ then  $$\|u^{r}\|_{C([0,T];H^2(\Omega))}
+\|\D u^r\|_{C([0,T];L^2(\Omega))}\le C  \|F\|_{C^\theta([0,T];L^2(\Omega))},$$ where $C$ depends on $\Omega$, $-\L$ and $a(t).$ 
\end{corollary}
\begin{proof}
For each $t\in[0,T],$ we have 
\begin{equation*}
\begin{split}
\|u^{r,+}(x,t;a)\|^2_{H^2(\Omega)}
&\le C \|-\L u^{r,+}\|^2_{L^2(\Omega)}
\le C \sum_{n=1}^\infty
|\lambda_nu^{r,+}_n(t;a)|^2\\
&\le C \sum_{n=1}^\infty
\left|\lambda_n \int_0^t F^+_n(\tau) (t-\tau)^{\alpha-1} 
E_{\alpha,\alpha}(-\lambda_n q_a(t-\tau)^\alpha){\rm d}\tau \right|^2\\
&\le C \sum_{n=1}^\infty
\left|\lambda_n \int_0^t |F^+_n(\tau) -F^+_n(t)|(t-\tau)^{\alpha-1} 
E_{\alpha,\alpha}(-\lambda_n q_a(t-\tau)^\alpha){\rm d}\tau\right|^2\\
&\quad +C \sum_{n=1}^\infty 
\left|F_n^+(t)\int_0^t \lambda_n(t-\tau)^{\alpha-1} 
E_{\alpha,\alpha}(-\lambda_n q_a(t-\tau)^\alpha){\rm d}\tau\right|^2.
\end{split}
\end{equation*}
The definition of $F_n^+(t)$ yields that  
$|F^+_n(\tau) -F^+_n(t)|\le |F_n(\tau) -F_n(t)|$;
Lemma \ref{mittag_derivative} gives 
$$0<\int_0^t \lambda_n(t-\tau)^{\alpha-1} 
E_{\alpha,\alpha}(-\lambda_n q_a(t-\tau)^\alpha){\rm d}\tau=
q_a^{-1}(1-E_{\alpha,1}(-\lambda_n q_a t^\alpha))<q_a^{-1}.$$
Hence, 
\begin{equation*}
\begin{split}
\|u^{r,+}(x,t;a)\|^2_{H^2(\Omega)}
&\le C \sum_{n=1}^\infty
\left|\lambda_n \int_0^t |F_n(\tau) -F_n(t)|(t-\tau)^{\alpha-1} 
E_{\alpha,\alpha}(-\lambda_n q_a(t-\tau)^\alpha){\rm d}\tau\right|^2\\
&\quad +C \sum_{n=1}^\infty 
\left|F_n(t)\right|^2.
\end{split}
\end{equation*}
By \cite[Lemma 3.4]{sakamoto2011initial}, we have 
$$\|u^{r,+}(x,t;a)\|^2_{H^2(\Omega)}
\le C \|F\|_{C^\theta([0,T];L^2(\Omega))}^2+C\| F(\cdot,t)\|_{L^2(\Omega)}^2,
\ t\in[0,T],
$$
which gives 
$$\|u^{r,+}\|_{C([0,T];H^2(\Omega))}\le C \|F\|_{C^\theta([0,T];L^2(\Omega))},$$
and the constant $C$ depends on $\Omega$, $-\L$ and $a(t).$ Similarly, we can show\\ $\|u^{r,-}\|_{C([0,T];H^2(\Omega))}\le C \|F\|_{C^\theta([0,T];L^2(\Omega))}.$

\par For $\D u^r,$ by \eqref{ODE_right}, we have 
$\D u^{r,+}=\sum_{n=1}^\infty [-\lambda_n a(t) u_n^{r,+}(t;a)+F_n^+(t)]\phi_n(x).$ 
Then for each $t\in[0,T],$
\begin{equation*}
\begin{split}
\|\D u^{r,+} \|^2_{L^2(\Omega)}&\le 
C\sum_{n=1}^\infty Q_a^2|\lambda_n  u_n^{r,+}(t;a)|^2+C\sum_{n=1}^\infty |F_n(t)|^2 \\
&\le C\sum_{n=1}^\infty |\lambda_n  u_n^{r,+}(t;a)|^2+C\|F(\cdot,t)\|_{L^2(\Omega)}^2.
\end{split}
\end{equation*}
From the above proof for $\|u^{r,+}\|^2_{H^2(\Omega)},$ it holds 
$$\|\D u^{r,+} \|^2_{L^2(\Omega)}
\le C \|F\|^2_{C^\theta([0,T];L^2(\Omega))}+C\|F(\cdot,t)\|_{L^2(\Omega)}^2,\ t\in[0,T],$$ 
which gives 
$$\|\D u^{r,+}\|_{C([0,T];L^2(\Omega))}
\le C \|F\|_{C^\theta([0,T];L^2(\Omega))}.$$
Analogously, we can show 
$\|\D u^{r,-}\|_{C([0,T];L^2(\Omega))}
\le C \|F\|_{C^\theta([0,T];L^2(\Omega))}.$

\par The estimates of $u^{r,+},\ u^{r,-},\ \D u^{r,+}$ and $\D u^{r,-}$ 
yield the desired result and complete this proof.
\end{proof}

\subsubsection{Regularity of $u^i$}
\par In this part we consider the regularity of $u^i.$ Just as in  
the regularity results for $u^r,$ we first state two lemmas which 
concern the positivity and monotonicity of $u^i,$ respectively.
\begin{lemma}\label{sign_u^i}
 With the representation \eqref{u^r u^i} and the fractional ODE
 \eqref{ODE_initial}, for each $\N+,$ $b_n \le (\ge)0$ implies that 
 $u^i_n(t;a)\le (\ge)0$ on $[0,T]$.
\end{lemma}
\begin{proof}
 This is a directly result of Corollary \ref{sign_eigenfunction}.
\end{proof}

\begin{lemma}\label{monotone_initial}
 Given $a_1, a_2\in C^{+}[0,T]$ with $a_1\le a_2$ on $[0,T]$, 
 for each $\N+,$ we have 
 \begin{equation*}
  \begin{cases}
   0\le u_n^i(t;a_2)\le u_n^i(t;a_1),\ \text{if}\ b_n\ge 0;\\
   u_n^i(t;a_1)\le u_n^i(t;a_2) \le 0,\ \text{if}\ b_n\le 0.
  \end{cases}
 \end{equation*}
\end{lemma}

\begin{proof}
 \par Fix $\N+,$ from the fractional ODE \eqref{ODE_initial}, 
 the functions $u^i_n(t;a_1)$ and $u^i_n(t;a_2)$ satisfy the following system
 \begin{equation*}
  \begin{cases}
   \D u_n^i(t;a_1)+\lambda_na_1(t)u_n^i(t;a_1)=0;\\
   \D u_n^i(t;a_2)+\lambda_na_2(t)u_n^i(t;a_2)=0;\\
   u_n^i(0;a_1)=u_n^i(0;a_2)=b_n.
  \end{cases}
 \end{equation*}
This gives 
 \begin{equation}\label{equality_2}
 \D w+\lambda_na_1(t)w(t)=\lambda_nu_n^i(t;a_2)(a_2(t)-a_1(t)),
 \ w(0)=0,
 \end{equation} 
 where $w(t)=u_n^i(t;a_1)-u_n^i(t;a_2)$.
 \par If $b_n\ge 0,$ Corollary \ref{sign_eigenfunction} 
 shows that $u_n^i(t;a_1),u_n^i(t;a_2)\ge 0.$  
 Also, Lemma \ref{sign_u^i} and 
 $a_1\le a_2$ ensures the right side of \eqref{equality_2} is nonnegative, 
 which together with Corollary \ref{sign_eigenfunction} implies $w\ge 0$,
 i.e. $0\le u_n^i(t;a_2)\le u_n^i(t;a_1).$
  The similar argument yields $u_n^i(t;a_1)\le u_n^i(t;a_2) \le 0$ 
 for the case $b_n\le 0.$
\end{proof}

\begin{lemma}[Regularity for $u^i$]\label{regularity_u^i}
 $$\|u^i\|_{L^2(0,T;H^2(\Omega))}+\|\D u^i\|_{L^2([0,T]\times\Omega)}
\le C T^{\frac{1-\alpha}{2}}\|u_0\|_{H^1(\Omega)}.$$
\end{lemma}

\begin{proof}
 \par Given $t\in [0,T],$ the direct calculation and Lemma \ref{monotone_initial}
 yield that 
 \begin{equation*}
 \begin{split}
  \|u^i(x,t;a)\|_{H^2(\Omega)}^2 
  &\le C\|-\L u^i(x,t;a)\|_{L^2(\Omega)}^2
  =C\|\sum_{n=1}^\infty \lambda_nu^i_n(t;a)\phi_n(x) \|_{L^2(\Omega)}^2\\
  &=C\sum_{n=1}^\infty|\lambda_nu^i_n(t;a)|^2
  \le C\sum_{n=1}^\infty|\lambda_nu^i_n(t;q_a)|^2.
  \end{split}  
 \end{equation*}
Recall that \cite{sakamoto2011initial} established the representation as
  $u^i_n(t;q_a)=b_n E_{\alpha,1}(-\lambda_nq_at^\alpha),\ \N+.$
Hence, by Lemma \ref{mittag_bound}, 
\begin{equation}\label{inequality_11}
\begin{split}
 \|u^i(x,t;a)\|_{H^2(\Omega)}^2
 &\le C\|-\L u^i(x,t;a)\|_{L^2(\Omega)}^2
 \le C\sum_{n=1}^\infty|\lambda_nb_n 
 E_{\alpha,1}(-\lambda_nq_at^\alpha)|^2\\
 &\le C\sum_{n=1}^\infty |\frac{1}{1+\lambda_nq_at^\alpha}|^2\lambda_n^2b_n^2
 = C\sum_{n=1}^\infty|\frac{(\lambda_nq_at^\alpha)^{\frac{1}{2}}}
 {1+\lambda_nq_at^\alpha}|^2t^{-\alpha}q_a^{-1}\lambda_nb_n^2\\
 &\le Ct^{-\alpha}\sum_{n=1}^\infty((-\L)^{\frac{1}{2}}u_0,\phi_n)^2
 \le Ct^{-\alpha}\|u_0\|_{H^1(\Omega)}^2,
\end{split}
\end{equation}
which leads to
$ \|u^i\|_{L^2(0,T;H^2(\Omega))}^2
 \le C\int_0^T t^{-\alpha}\|u_0\|_{H^1(\Omega)}^2 {\rm d}t
 =CT^{1-\alpha}\|u_0\|_{H^1(\Omega)}^2,$
i.e.
\begin{equation}\label{inequality_12}
 \|u^i\|_{L^2(0,T;H^2(\Omega))}
\le C T^{\frac{1-\alpha}{2}}\|u_0\|_{H^1(\Omega)}.
\end{equation}

\par For the estimate of $\D u^i(x,t;a),$ 
\eqref{u^r u^i} and \eqref{ODE_initial} yield  
$$\D u^i(x,t;a)=\sum_{n=1}^{\infty} {\D} u^i_n(t;a)\phi_n(x)
=-\sum_{n=1}^\infty \lambda_n a(t) u^i_n(t;a)\phi_n(x),$$
which together with \eqref{q_a Q_a} gives
\begin{equation*}
\begin{split}
\|\D u^i(x,t;a)\|_{L^2(\Omega)}^2
&\le Q_a^2 \sum_{n=1}^\infty |\lambda_n u^i_n(t;a)|^2\\
&=Q_a^2 \|-\L u^i(x,t;a)\|_{L^2(\Omega)}^2
\le Ct^{-\alpha}\|u_0\|_{H^1(\Omega)}^2,\ t\in[0,T],
\end{split}
\end{equation*}
where the last inequality follows from \eqref{inequality_11}.
This result implies that 
\begin{equation*}
 \|\D u^i(x,t;a)\|_{L^2([0,T]\times\Omega)}^2
 =\int_0^T \|\D u^i(x,t;a)\|_{L^2(\Omega)}^2 {\rm d}t
 \le CT^{1-\alpha}\|u_0\|_{H^1(\Omega)}^2,
\end{equation*}
i.e. $\|\D u^i\|_{L^2([0,T]\times\Omega)}
 \le  C T^{\frac{1-\alpha}{2}}\|u_0\|_{H^1(\Omega)},$ which together 
 with \eqref{inequality_12} completes the proof.
\end{proof}

\par Moreover, with a stronger condition on $u_0,$ such as assuming  
$u_0\in H^2(\Omega)\cap H^1_0(\Omega),$
we can deduce the $C$-regularity estimate of $u^i.$

\begin{corollary}\label{regularity_u^i_H^2}
 With Assumption \ref{assumption_direct} and 
 $u_0\in H^2(\Omega)\cap H^1_0(\Omega),$ then  
 \begin{equation*}
  \|u^i\|_{C([0,T];H^2(\Omega))}+\|\D u^i\|_{C([0,T];L^2(\Omega))}
 \le C\|u_0\|_{H^2(\Omega)}.
\end{equation*}
\end{corollary}

\begin{proof}
\par Lemma \ref{mittag_bound} yields that 
\begin{equation*}\label{inequality_15}
\sum_{n=1}^\infty|\lambda_nb_n 
 E_{\alpha,1}(-\lambda_nq_at^\alpha)|^2
 \le C\sum_{n=1}^\infty|\lambda_nb_n|^2
 =C\|-\L u_0\|_{L^2(\Omega)}^2 \le C \|u_0\|_{H^2(\Omega)}^2,\ t\in[0,T];
\end{equation*}
meanwhile, the following estimates have been shown in 
 the proof of Theorem \ref{regularity_u^i}
\begin{equation*}\label{inequality_14}
 \begin{cases}
  \|u^i(x,t;a)\|_{H^2(\Omega)}^2
 \le C\|-\L u^i(x,t;a)\|_{L^2(\Omega)}^2
 \le C\sum_{n=1}^\infty|\lambda_nb_n 
 E_{\alpha,1}(-\lambda_nq_at^\alpha)|^2,\\
 \|\D u^i(x,t;a)\|_{L^2(\Omega)}^2
\le Q_a^2 \sum_{n=1}^\infty |\lambda_n u^i_n(t;a)|^2
=C \|-\L u^i(x,t;a)\|_{L^2(\Omega)}^2.
 \end{cases}
\end{equation*}
Hence, it holds that 
\begin{equation*}
 \|u^i(x,t;a)\|_{H^2(\Omega)}+\|\D u^i(x,t;a)\|_{L^2(\Omega)}
 \le C\|u_0\|_{H^2(\Omega)},\ t\in[0,T],
\end{equation*}
which leads to the claimed result.
\end{proof}

\subsection{Main theorem for the direct problem}

\par The main theorem for the direct problem follows from Theorem 
\ref{existenceuniqueness}, Lemmas  \ref{regularity_u^r} and \ref{regularity_u^i}, 
Corollaries \ref{Cregularity_u^r} and \ref{regularity_u^i_H^2}, and the relation 
$u(x,t;a)=u^r(x,t;a)+u^i(x,t;a).$

\begin{theorem}[Main theorem for the direct problem]\label{main_direct}
 \par Let Assumption \ref{assumption_direct} be valid, then under Definition 
 \ref{weak solution}, there exists a unique weak solution $u(x,t;a)$ 
 of FDE \eqref{fde} with the spectral representation 
 \eqref{solution} and the following regularity estimates:
 $$\|u\|_{L^2(0,T;H^2(\Omega))}+\|\D u\|_{L^2([0,T]\times\Omega)}
\le C(\|F\|_{L^2([0,T]\times\Omega)}
+T^{\frac{1-\alpha}{2}}\|u_0\|_{H^1(\Omega)}).$$
Moreover, if the conditions $u_0\in H^2(\Omega)\cap H^1_0(\Omega)$ and $F\in C^\theta([0,T];L^2(\Omega)),\ 0<\theta<1$ are added, we have: 
 \begin{equation*}
  \|u\|_{C([0,T];H^2(\Omega))}+\|\D u\|_{C([0,T];L^2(\Omega))}
 \le C(\|F\|_{C^\theta([0,T];L^2(\Omega))}+\|u_0\|_{H^2(\Omega)}).
 \end{equation*}
\end{theorem}

\section{Inverse Problem--Reconstruction of the diffusion coefficient $a(t)$}

\par In this section, we discuss how to recover the coefficient $a(t)$ 
through the output flux data 
$$a(t)\frac{\partial u}{\n}(x_0,t;a)=g(t),\ x_0\in \partial\Omega.$$ 
All cross the inverse problem work, the operator $-\L$ is assumed to satisfy the condition \eqref{sobolev assumption}, then the expression $\frac{\partial \phi_n}{\n}(x_0)$ makes sense.
We only consider this reconstruction in the space $C^+[0,T],$ which 
can be regarded as the admissible set for $a(t).$
To this end, we introduce an operator $K,$ which will be shown to have 
a fixed point consisting of the desired coefficient $a(t).$ 

\subsection{Operator $K$}

\par The operator $K$ is defined as 
$$K \psi(t):=\frac{g(t)}{\frac{\partial u}{\n}(x_0,t;\psi)}
=\frac{g(t)}{\sum\limits_{n=1}^\infty u_n(t;\psi)
\frac{\partial \phi_n}{\n}(x_0)},
\ t\in [0,T]$$ 
with domain 
$$\DK:=\{\psi\in C^{+}[0,T]:\psi(t)\ge g(t)\Big[\frac{\partial u_0}{\n}(x_0)
+ \I[\frac{\partial F}{\n}(x_0,t)]\Big]^{-1},\ t\in[0,T] \}.$$

\par To analyze $K$, we make the following assumptions. 
\begin{assumption}\label{assumption_inverse}
 $u_0,$ $F$ and $g$ should satisfy the following restrictions:
 \begin{itemize}
  \item [(a)] $u_0\in H^3(\Omega)\cap H_0^1(\Omega)$ with 
  $b_n:=(u_0,\phi_n)\ge 0,\ \N+;$
  \item [(b)] $\exists \theta\in (0,1)$ s.t. $F(x,t)\in C^\theta([0,T];H^3(\Omega)\cap H_0^1(\Omega))$ with 
  $F_n(t):=(F(\cdot,t),\phi_n)\ge 0$ on $[0,T]$ for each $\N+;$
  \item [(c)] $\exists N\in \mathbb{N}^{+}$ s.t. 
  $\frac{\partial \phi_N}{\n}(x_0)>0,$ $b_N>0$ and $F_N(t)>0$ on $[0,T];$ 
  \item [(d)] $g\in C^+[0,T].$
 \end{itemize}
\end{assumption}

\par The next remark shows that the equality in the definition of $K$ is valid. 
\begin{remark}\label{partial derivative of u}
Given $\psi\in C^+[0,T]$ and for each $t\in[0,T],$ by the proofs of Corollaries \ref{Cregularity_u^r} and \ref{regularity_u^i_H^2}, we have 
\begin{equation*}
\begin{split}
\|u^{r,+}(x,t;\psi)\|^2_{H^3(\Omega)}
&\le C \|(-\L)^{3/2}u^{r,+}\|^2_{L^2(\Omega)}
\le C \sum_{n=1}^\infty
|\lambda_n^{3/2}u^{r,+}_n(t;\psi)|^2\\
&\le C \sum_{n=1}^\infty
\left|\lambda_n^{3/2} \int_0^t F^+_n(\tau) (t-\tau)^{\alpha-1} 
E_{\alpha,\alpha}(-\lambda_n q_\psi(t-\tau)^\alpha){\rm d}\tau \right|^2\\
&\le C \sum_{n=1}^\infty
\left|\lambda_n \int_0^t \lambda_n^{1/2}|F^+_n(\tau) -F^+_n(t)|(t-\tau)^{\alpha-1} 
E_{\alpha,\alpha}(-\lambda_n q_\psi(t-\tau)^\alpha){\rm d}\tau\right|^2\\
&\quad +C \sum_{n=1}^\infty 
\left|\lambda_n^{1/2}F_n^+(t)(1-E_{\alpha,1}(-\lambda_n q_\psi t^\alpha))\right|^2\\
&\le C \sum_{n=1}^\infty
\left|\lambda_n \int_0^t \lambda_n^{1/2}|F_n(\tau) -F_n(t)|(t-\tau)^{\alpha-1} 
E_{\alpha,\alpha}(-\lambda_n q_\psi(t-\tau)^\alpha){\rm d}\tau\right|^2\\
&\quad +C \sum_{n=1}^\infty 
\left|\lambda_n^{1/2}F_n(t)(1-E_{\alpha,1}(-\lambda_n q_\psi t^\alpha))\right|^2\\
&\le C \|(-\L)^{1/2} F\|_{C^\theta([0,T];L^2(\Omega))}^2+C\|(-\L)^{1/2} F(\cdot,t)\|_{L^2(\Omega)}^2\\
&\le C \| F\|_{C^\theta([0,T];H^1(\Omega))}^2+C\|F(\cdot,t)\|_{H^1(\Omega)}^2
\end{split}
\end{equation*}
and 
$$\|u^{r,-}(x,t;\psi)\|^2_{H^3(\Omega)}
\le C \| F\|_{C^\theta([0,T];H^1(\Omega))}^2+C\|F(\cdot,t)\|_{H^1(\Omega)}^2,$$
which give 
$\|u^r\|_{C([0,T];H^3(\Omega))}
\le C\| F\|_{C^\theta([0,T];H^1(\Omega))}$;
\begin{equation*}
\begin{split}
\|u^i(x,t;\psi)\|^2_{H^3(\Omega)}
&\le C \|(-\L)^{3/2}u^i\|^2_{L^2(\Omega)}
\le C \|\sum_{n=1}^\infty\lambda_n^{3/2}u_n^i(t;\psi)\phi_n(x)\|^2_{L^2(\Omega)}\\
&\le C\sum_{n=1}^\infty|\lambda_n^{3/2}b_n 
E_{\alpha,1}(-\lambda_nq_\psi t^\alpha)|^2
\le C\sum_{n=1}^\infty|\lambda_n^{3/2}b_n|^2\\
&=C\|(-\L)^{3/2} u_0\|_{L^2(\Omega)}^2 \le C \|u_0\|_{H^3(\Omega)}^2,
\end{split}
\end{equation*}
which gives $\|u^i\|_{C([0,T];H^3(\Omega))}\le C \|u_0\|_{H^3(\Omega)}.$
Combining the above two results yields that 
$$\|u\|_{C([0,T];H^3(\Omega))}\le C (\| F\|_{C^\theta([0,T];H^1(\Omega))}+\|u_0\|_{H^3(\Omega)})<\infty,$$
which means for each $t\in [0,T],$ 
$\|u\|_{H^3(\Omega)}<\infty.$ Recall that $\Omega\subset R^n, n=1,2,3$, then the Sobolev Embedding Theorem gives 
$$u(x,t;\psi)=\sum_{n=1}^\infty u_n(t;\psi)\phi_n(x)\in C^1(\overline{\Omega})\ \text{for each}\ t\in [0,T].$$ 
Hence, 
$\sum\limits_{n=1}^\infty u_n(t;\psi)
\frac{\partial \phi_n}{\n}(x_0)$ is well-defined and 
$$\frac{\partial u}{\n}(x_0,t;\psi)=\sum\limits_{n=1}^\infty u_n(t;\psi)
\frac{\partial \phi_n}{\n}(x_0),\quad t\in[0,T].$$
\end{remark}

\par The following two remarks will explain the reasonableness and reason 
for Assumption \ref{assumption_inverse}.
\begin{remark}\label{resonable_assumption_inverse}
 \par For the inverse problem, the right-hand side function $F(x,t)$
 and the initial condition $u_0(x)$ are input data, which, 
 at least in some circumstance, can be assumed to be controlled. 
 Even though Assumption \ref{assumption_inverse} 
 $(a),$ $(b)$ and $(c)$ appear restrictive, it is not hard to construct 
 functions that satisfy them. For example, in $(a)$ if $u_0=c\phi_k$ for 
 some $c>0,$ then Assumption \ref{assumption_inverse} $(a)$ will be 
 satisfied. This will also be true if $u_0=\sum_{k=1}^{M} c_k\phi_k$ 
 with all $c_k>0.$ Similarly, $(b)$ is satisfied if $F(x,t)$ is also a 
 linear combination of $\{\phi_n:\N+\}$ with positive coefficients. 
 For $(c),$ by the completeness of $\{\phi_n:\N+\}$ in $L^2(\Omega),$ there should exist 
 $N\in \mathbb{N}^{+}$ s.t. $\frac{\partial \phi_N}{\n}(x_0)>0.$ Otherwise, for each 
 $\psi\in H^3(\Omega)\subset L^2(\Omega),$ $\frac{\partial \psi}{\n}(x_0)=0$ and obviously 
 it is incorrect. Then for this $N,$ we only need to set the 
 coefficients of $u_0$ and $F$ upon $\phi_N$ be strictly positive. 
  \par The output flux data $g(t),$ it is not under our control. 
  However, if there exists $a\in C^+[0,T]$ s.t. $a(t)\frac{\partial u}{\n}(x_0,t;a)=g(t),$  
 Assumption \ref{assumption_inverse} $(a),$ $(b)$ and Corollary 
 \ref{sign_eigenfunction} yield that $u_n(t;a)\ge 0;$ 
  \eqref{sign_eigenfunction_derivative} gives
 $\frac{\partial \phi_n}{\n}(x_0)\ge 0, \N+;$
 Assumption \ref{assumption_inverse} $(c)$ ensures 
 $\frac{\partial \phi_N}{\n}(x_0)>0$ and 
 $u_N(t;a)>0$ on $[0,T],$ where the proof can be seen in 
 Lemma \ref{well_definedness}. 
 Consequently, $$\frac{\partial u}{\n}(x_0,t;a)
 =\sum\limits_{n=1}^\infty u_n(t;a)\frac{\partial \phi_n}{\n}(x_0)
 \ge u_N(t;a)\frac{\partial \phi_N}{\n}(x_0)>0,\ t\in[0,T].$$ 
 This together with $a\in C^+[0,T]$ gives that $g>0.$
 The continuity of $g$ follows from the ones of $a$ and $u_n(t;a),\ \N+,$ 
 which are derived from the admissible set $C^+[0,T]$ and 
 Theorem \ref{existenceuniqueness}, respectively.
 Therefore, Assumption \ref{assumption_inverse} $(d)$ is 
 reasonable and can be attained.
\end{remark}

\begin{remark}\label{remark_assumption_inverse}
 The well-definedness of the domain $\DK$ is guaranteed by 
 Assumption \ref{assumption_inverse} $(a)$, $(b)$, $(c)$ and $(d)$ in the 
 sense that the $H^3$-regularity of $u_0,$ $F$ and the Sobolev Embedding Theorem support that $\frac{\partial u_0}{\n}(x_0)$ and $\frac{\partial F}{\n}(x_0,t)$ are well defined, and 
 the dominator of the lower bound of $\DK$ 
 $$
 \frac{\partial u_0}{\n}(x_0)+\I[\frac{\partial F}{\n}(x_0,t)]
 =\sum_{n=1}^{\infty} (b_n+\I F_n) \frac{\partial \phi_n}{\n}(x_0)
 \ge (b_N+\I F_N) \frac{\partial \phi_N}{\n}(x_0)>0
 $$
 on $[0,T].$ Recall that the numerator $g>0,$ so that the lower bound 
 $g(t)\Big[\frac{\partial u_0}{\n}(x_0)
 +\I[\frac{\partial F}{\n}(x_0,t)]\Big]^{-1}>0,$ which 
 gives that $\DK$ is a subspace of $C^+[0,T].$ 
 Also, $F(x,t)\in C^\theta([0,T];H^3(\Omega)\cap H_0^1(\Omega))$ yields that 
 $F_N(t)$ is continuous on $[0,T],$ so is 
 $(b_N+\I F_N) \frac{\partial \phi_N}{\n}(x_0).$ Then $\exists C>0$
 s.t. $(b_N+\I F_N) \frac{\partial \phi_N}{\n}(x_0)>C>0,$ which leads to 
 the dominator $$\frac{\partial u_0}{\n}(x_0)
 +\I[\frac{\partial F}{\n}(x_0,t)]>C>0\  \text{on}\ [0,T].$$ 
 The strict positivity of the dominator avoids $\DK$ degenerating 
 to an empty set. 
 \par In order to show the well-definedness of $K,$  
 Assumption \ref{assumption_inverse} $(a),$ $(b)$ and $(c)$ will 
 be used. Furthermore, Assumption \ref{assumption_inverse} 
 $(a)$ and $(b)$ are crucial to build the monotonicity of operator $K$; 
 meanwhile, Assumption \ref{assumption_inverse} $(c)$ is stated for 
 the uniqueness of fixed points of $K$.
\end{remark}

\par For the operator $K$, we have the following lemmas. 
\begin{lemma}\label{well_definedness}
 The operator $K$ is well-defined.
\end{lemma}
\begin{proof}
 \par For each $\psi\in \DK,$ Theorem \ref{existenceuniqueness} ensures that 
  there exists a unique $u_n(t;\psi)$ for $\N+,$ 
 which implies the existence and uniqueness of $K\psi.$
 
 \par Then it is suffice to show the dominator $\sum\limits_{n=1}^\infty u_n(t;\psi)
\frac{\partial \phi_n}{\n}(x_0)>0$ on $[0,T].$
With \eqref{ODE}, Lemma \ref{sign_eigenfunction} and Assumption 
\ref{assumption_inverse} $(a)$ and $(b),$ we have 
$u_n(t;\psi)\ge 0$ on $[0,T],$ which together with 
$\frac{\partial \phi_n}{\n}(x_0)\ge0$ gives 
$\sum\limits_{n=1}^\infty u_n(t;\psi)\frac{\partial \phi_n}{\n}(x_0)
\ge u_N(t;\psi)\frac{\partial \phi_N}{\n}(x_0).$
Due to the assumption $\frac{\partial \phi_N}{\n}(x_0)>0,$ 
we claim that $u_N(t;\psi)>0.$ 
Assume not, i.e. $\exists t_0\in [0,T]$ s.t. $u_N(t_0;\psi)\le 0.$ 
The result $u_N(t;\psi)\ge 0$ yields that $u_N(t_0;\psi)=0$ so that 
$u_N(t;\psi)$ attains its minimum at $t=t_0.$ 
$u_N(0;\psi)=b_N>0$ implies $t_0\ne 0,$ i.e. $t_0\in(0,T].$ 
Then Lemma \ref{extreme}, $u_N(t_0;\psi)=0$ and the ODE \eqref{ODE}
show that $^C\!D_{t}^{\alpha}u_N(t_0;\psi)=F_N(t_0)\le 0,$ 
which contradicts with Assumption \ref{assumption_inverse} $(c)$ 
and confirms the claim. Hence, 
$$\sum\limits_{n=1}^\infty u_n(t;\psi)\frac{\partial \phi_n}{\n}(x_0)
\ge u_N(t;\psi)\frac{\partial \phi_N}{\n}(x_0)>0,$$ 
which completes the proof.
\end{proof}

\begin{lemma}\label{itself}
 $K$ maps $\DK$ into $\DK.$
\end{lemma}
\begin{proof}
 \par Given $\psi\in \DK.$ The continuity of $K\psi$ follows from the continuity 
 of $u_n(t;\psi)$ for each $\N+$ and the continuity of $g$, which are 
 established by Theorem \ref{existenceuniqueness} and Assumption 
 \ref{assumption_inverse} $(d)$ respectively.
 
 \par For each $\N+,$ \eqref{ODE} ensures $u_n(t;\psi)$ satisfies 
 \begin{equation*}
  \D u_n(t;\psi)+\lambda_n \psi(t) u_n(t;\psi)=F_n(t),
  \ u_n(0;\psi)=b_n.
 \end{equation*}
Taking $\I$ on both sides of the above ODE and using 
 Lemma \ref{I_alpha} yield that 
 \begin{equation*}
  u_n(t;\psi)+\lambda_n \I[\psi(t) u_n(t;\psi)]=\I F_n+b_n.
 \end{equation*}
 From the proof of Lemma \ref{well_definedness}, we have 
 $u_n(t;\psi)\ge 0$ on $[0,T],$ which together with 
 $\lambda_n>0,$ the positivity of $\psi$ and the definition of 
 $\I$ yields that $\lambda_n \I[\psi(t) u_n(t;\psi)]\ge 0.$
 Since $u_n(t;\psi)\ge 0$ and 
 $\lambda_n \I[\psi(t) u_n(t;\psi)]\ge 0,$ 
 we deduce that $0\le u_n(t;\psi)\le \I F_n+b_n$ on $[0,T].$
 Hence, with $\frac{\partial \phi_n}{\n}(x_0)\ge 0$ and 
 the smoothness assumptions $u_0\in H^3(\Omega)\cap H_0^1(\Omega),\ 
F\in C^\theta([0,T];H^3(\Omega)\cap H_0^1(\Omega))$ 
stated in Assumption \ref{assumption_inverse} $(a)$ and $(b)$ respectively,
 the following inequality holds
 \begin{equation*}\label{inequality_16}
  \sum_{n=1}^\infty u_n(t;\psi)\frac{\partial \phi_n}{\n}(x_0)
  \le \sum_{n=1}^\infty (\I F_n+b_n) \frac{\partial \phi_n}{\n}(x_0)
  =\frac{\partial u_0}{\n}(x_0)+\I[\frac{\partial F}{\n}(x_0,t)],
 \end{equation*}
which together with $g>0$ yields that 
$$K\psi(t) = \frac{g(t)}{\sum\limits_{n=1}^\infty u_n(t;\psi)
\frac{\partial \phi_n}{\n}(x_0)}\ge
g(t)\Big[\frac{\partial u_0}{\n}(x_0)
+ \I[\frac{\partial F}{\n}(x_0,t)]\Big]^{-1}>0,\ t\in[0,T],
$$
where the last inequality follows from Remark \ref{remark_assumption_inverse}.
The above result and the continuity of $K\psi$ lead to 
$K\psi \in \DK,$ which is the expected result.
\end{proof}

\subsection{Monotonicity}
\par In this part, we show the monotonicity of the operator $K$.
\begin{theorem}[Monotonicity]\label{monotonicity}
 Given $a_1,a_2\in \DK$ with $a_1\le a_2,$ then 
 $Ka_1\le Ka_2$ on $[0,T].$
\end{theorem}
\begin{proof}
\par Pick $\N+,$ due to \eqref{ODE}, $u_n(t;a_1)$ and $u_n(t;a_2)$ 
satisfy 
\begin{equation*}
\begin{cases}
\D u_n(t;a_1)+\lambda_n a_1(t) u_n(t;a_1)=F_n(t),
  \ u_n(0;a_1)=b_n;\\
  \D u_n(t;a_2)+\lambda_n a_2(t) u_n(t;a_2)=F_n(t),
  \ u_n(0;a_2)=b_n,
\end{cases}  
\end{equation*}
which together with $a_1\le a_2$ and Lemma \ref{sign} yields 
\begin{equation}\label{w}
\D w+\lambda_na_1(t)w(t)=\lambda_nu_n(t;a_2)(a_2(t)-a_1(t))\ge 0,
 \ w(0)=0,
\end{equation}
where $w(t)=u_n(t;a_1)-u_n(t;a_2).$
Applying Lemma \ref{sign} to the above ODE yields that 
$w\ge 0,$ i.e. $u_n(t;a_1)\ge u_n(t;a_2)\ge0,$ which together with 
assumption \eqref{sign_eigenfunction_derivative} leads to 
$$\sum\limits_{n=1}^\infty u_n(t;a_1)\frac{\partial \phi_n}{\n}(x_0)
\ge\sum\limits_{n=1}^\infty u_n(t;a_2)\frac{\partial \phi_n}{\n}(x_0) 
> 0,\ t\in[0,T].$$
Therefore, with the condition $g>0$ stated in 
Assumption \ref{assumption_inverse} $(d)$, 
$$Ka_1(t)=\frac{g(t)}{\sum\limits_{n=1}^\infty u_n(t;a_1)
\frac{\partial \phi_n}{\n}(x_0)}
\le \frac{g(t)}{\sum\limits_{n=1}^\infty u_n(t;a_2)
\frac{\partial \phi_n}{\n}(x_0)}=Ka_2(t),\ t\in[0,T],$$ 
which completes this proof.
\end{proof}

\subsection{Uniqueness}
\par In order to show the uniqueness, we state two lemmas. 
\begin{lemma}\label{uniqueness_monotone}
 If $a_1,a_2\in \DK$ are both fixed points of $K$ with $a_1\le a_2,$ 
 then $a_1\equiv a_2$.
\end{lemma}
\begin{proof}
 \par Pick a fixed point $a(t),$ then 
 $$a(t)\sum\limits_{n=1}^\infty u_n(t;a)
\frac{\partial \phi_n}{\n}(x_0)
=\sum\limits_{n=1}^\infty a(t)u_n(t;a)
\frac{\partial \phi_n}{\n}(x_0)=g(t),$$
which gives 
\begin{equation}\label{equality_3}
\sum\limits_{n=1}^\infty \I[a(t)u_n(t;a)]
\frac{\partial \phi_n}{\n}(x_0)=\I g
\end{equation}
by taking $\I$ on both sides.
Similarly, taking $\I$ on the both sides of \eqref{ODE}
and applying Lemma \ref{I_alpha} yield that 
\begin{equation*}\label{equality_4}
 \I[a(t)u_n(t;a)]=\lambda_n^{-1}\I F_n+\lambda_n^{-1}b_n
-\lambda_n^{-1}u_n(t;a),\ \N+,
\end{equation*}
which together with \eqref{equality_3} generates 
\begin{equation}\label{equality_5}
 \sum\limits_{n=1}^\infty \lambda_n^{-1}u_n(t;a)
 \frac{\partial \phi_n}{\n}(x_0)
 =\sum\limits_{n=1}^\infty \lambda_n^{-1}(\I F_n+b_n)
 \frac{\partial \phi_n}{\n}(x_0)-\I g.
\end{equation}
In \eqref{equality_5}, the convergence of the two series in $C[0,T]$ is supported by Assumption \ref{assumption_inverse}, Remark \ref{partial derivative of u} and the fact that $0<\lambda_1\le\lambda_2\le\cdots.$ 

\par Given two fixed points $a_1,a_2$ with $a_1\le a_2,$ then $a_1$ and $a_2$ 
should satisfy \eqref{equality_5} simultaneously, which gives 
\begin{equation}\label{equality_6}
 \sum\limits_{n=1}^\infty \lambda_n^{-1}
 \frac{\partial \phi_n}{\n}(x_0)(u_n(t;a_1)-u_n(t;a_2))=0.
\end{equation}
In the proof of Theorem \ref{monotonicity}, we have shown that  
$u_n(t;a_1)\ge u_n(t;a_2)\ge0$. Also recall that 
$\lambda_n^{-1} \frac{\partial \phi_n}{\n}(x_0)\ge 0,\ \N+,$ 
then $\lambda_n^{-1}\frac{\partial \phi_n}{\n}(x_0)
(u_n(t;a_1)-u_n(t;a_2))\ge 0$ on $[0,T]$ for $\N+.$
Hence, \eqref{equality_6} implies that 
$$\lambda_n^{-1}\frac{\partial \phi_n}{\n}(x_0)
(u_n(t;a_1)-u_n(t;a_2))= 0,\ t\in[0,T],\ \N+.$$
Let $n=N,$ $\lambda_N^{-1}\frac{\partial \phi_N}{\n}(x_0)>0$
gives $u_N(t;a_1)\equiv u_N(t;a_2)$ on $[0,T].$
Set $w(t)=u_N(t;a_1)-u_N(t;a_2)=0.$ Then \eqref{w} yields that 
$$
0={\D} w+\lambda_Na_1(t)w(t)=\lambda_Nu_N(t;a_2)(a_2(t)-a_1(t)),
$$
i.e. $u_N(t;a_2)(a_2(t)-a_1(t))\equiv 0$ on $[0,T]$; while the proof of Lemma \ref{well_definedness} yields that $u_N(t;a_2)>0.$ Hence, we have  
$a_1= a_2$ on $[0,T],$ which completes the proof.
\end{proof}

\par Before showing uniqueness, we introduce a successive iteration procedure 
which will generate a sequence converging to a fixed point if it exists. Set 
$$\overline{a}_0(t)=g(t)\Big[\frac{\partial u_0}{\n}(x_0)
+ \I[\frac{\partial F}{\n}(x_0,t)]\Big]^{-1},\ 
\overline{a}_{n+1}=K\overline{a}_n,\ \NN.$$ Then this iteration reproduces 
a sequence $\{\overline{a}_n:\NN\}$ which is contained by 
$\DK$ due to Lemma \ref{itself}.

\begin{lemma}\label{uniqueness_series}
 If there exists a fixed point $a(t)\in \DK$ of operator $K,$ then the 
 sequence $\{\overline{a}_n:\NN\}$ will converge to $a(t).$
\end{lemma}
\begin{proof}
 \par  $\overline{a}_0$ is the lower bound of $\DK$ and 
 $\{\overline{a}_n:\NN\}\subset\DK$ yield that 
 $\overline{a}_0\le \overline{a}_1.$ Using Theorem \ref{monotonicity}, 
 we have $\overline{a}_1=K\overline{a}_0\le K\overline{a}_1=\overline{a}_2,$ 
 i.e. $\overline{a}_1\le \overline{a}_2.$ The same argument 
 gives $\overline{a}_2=K\overline{a}_1\le K\overline{a}_2=
 \overline{a}_3.$ Continue this process, we can deduce 
 $\overline{a}_0\le\overline{a}_1\le\overline{a}_2\le \dots,$ 
 which means $\{\overline{a}_n:\NN\}$ is increasing. 
 Since the results that $\overline{a}_0$ is the lower bound of $\DK$ 
 and $a(t)\in \DK,$ it holds $\overline{a}_0\le a.$ 
 Applying Theorem \ref{monotonicity} 
 to this inequality, we obtain $\overline{a}_1=K\overline{a}_0\le Ka=a,$ 
 i.e. $\overline{a}_1\le a.$
 This argument generates 
  $\overline{a}_n\le a,\ \NN,$ which means $a(t)$ is an upper bound 
 of $\{\overline{a}_n:\NN\}.$
 
 \par We have proved $\{\overline{a}_n:\NN\}$ is an increasing sequence 
 in $\DK$ with an upper bound $a(t),$ which leads to $\{\overline{a}_n:\NN\}$
 is convergent in $\DK$ and the limit is smaller than $a(t).$ 
 Denote the limit of $\{\overline{a}_n:\NN\}$ by $\overline{a}$. 
 We have $\overline{a}\in \DK,$ $\overline{a}\le a$ and 
 $\overline{a}$ is a fixed point of $K$ in $\DK.$ 
 Hence, Lemma \ref{uniqueness_monotone} yields 
 $\overline{a}=a,$ which is the desired result. 
 \end{proof}

 \par Now, we are able to prove the uniqueness of fixed points of $K$.
 \begin{theorem}[Uniqueness]\label{uniqueness}
 There is at most one fixed point of $K$ in $\DK.$
\end{theorem}
\begin{proof}
 \par Let $a_1, a_2\in \DK$ be both fixed points of $K.$ 
 Lemma \ref{uniqueness_series} implies that 
 $\overline{a}_n\to a_1$ and $\overline{a}_n\to a_2,$ which 
 leads to $a_1=a_2$ and completes this proof.
\end{proof}

\subsection{Existence}
\par Assumption \ref{assumption_inverse} is not sufficient to deduce the 
existence of the fixed points of $K$ since $\DK$ has 
no upper bound so that an increasing  sequence in $\DK$ may not be convergent.
In this part, we discuss the existence of fixed points, by providing  
some extra conditions.
\begin{assumption}\label{assumption_inverse_existence}
Additional assumptions on $u_0$, $F$ and $g$:
\begin{itemize}
 \item [(a)] $-\L u_0\in H^3(\Omega)\cap H_0^1(\Omega);$
 \item [(b)] $F(x,t)=-\L u_0(x)\cdot f(t)$ s.t. $f\in C^\theta[0,T], 0<\theta<1$ and $f(t)\ge g(t)
 \big[\frac{\partial u_0}{\n}(x_0)\big]^{-1}$ on $[0,T].$
\end{itemize}
\end{assumption}

\begin{remark}
Assumption \ref{assumption_inverse_existence} is set up to make sure that 
$F(x,t)=-\L u_0(x)\cdot f(t)\in C^\theta([0,T];H^3(\Omega)\cap H_0^1(\Omega))$, so that 
$F(x,t)$ also satisfies Assumption \ref{assumption_inverse}.

Fix $u_0$ and $f,$ if the measured data $g$ does not satisfy 
Assumption \ref{assumption_inverse_existence} $(b),$ then we can modify 
$u_0$ by increasing the value of $u_0$ in a very small neighborhood 
of the point $x_0$ so that the value of $\frac{\partial u_0}{\n}(x_0)$ 
becomes larger. Meanwhile, since $u_0$ is changed in a small domain, 
the coefficients $\{b_n:\N+\}$ only vary slightly, so do $u_n(t;a)$ and $u(x,t;a).$ 
Hence, $\frac{\partial u}{\n}(x_0,t;a)$ and $g(t)$ will not appear 
a significant change that can violate Assumption 
\ref{assumption_inverse_existence} $(b).$
\end{remark}

\par Define the subspace $\DK'$ of $\DK$ as 
\begin{equation*}
\begin{split}
 \DK':=\Big\{\psi\in C^{+}[0,T]:\ & g(t)\Big[\frac{\partial u_0}{\n}(x_0)
+ \I[\frac{\partial F}{\n}(x_0,t)]\Big]^{-1}\\
&\le \psi(t)\le g(t)\Big[\frac{\partial u_0}{\n}(x_0)\Big]^{-1},
\ t\in[0,T]\Big\}.
\end{split}
\end{equation*} 
We have proved the lower bound of $\DK'$ is positive in Remark 
 \ref{remark_assumption_inverse} and clearly the upper bound 
 of $\DK'$ is larger than the lower bound. Consequently, 
 $\DK'$ is well-defined. 
 
\par The next lemma concerns the range of $K$ with domain $\DK'$.
\begin{lemma}\label{itself_existence}
 With Assumptions \ref{assumption_inverse} and 
 \ref{assumption_inverse_existence}, K maps $\DK'$ into $\DK'$.
\end{lemma}
\begin{proof}
 \par Given $\psi \in \DK'$, we have proved 
 $K \psi \in C^{+}[0,T]$ and 
 $$K\psi(t) \ge g(t)\Big[\frac{\partial u_0}{\n}(x_0)
+ \I[\frac{\partial F}{\n}(x_0,t)]\Big]^{-1},\ t\in[0,T]$$ 
in the proof of Lemma \ref{itself}, so that 
it is sufficient to show $K\psi \le 
g(t)\big[\frac{\partial u_0}{\n}(x_0)\big]^{-1}$ on $[0,T].$

\par For each $\N+,$ let $w_n(t;\psi)=u_n(t;\psi)-b_n$, 
\eqref{ODE} yields the following ODE by direct calculation
$$
\D w_n(t;\psi)+\lambda_n \psi(t)w_n(t;\psi)=\lambda_nb_n(f(t)-\psi(t))\ge0,
\ w_n(0,\psi)=0,
$$
where $\lambda_nb_n(f(t)-\psi(t))\ge 0$ follows from 
the fact $\psi(t) \le g(t)\big[\frac{\partial u_0}{\n}(x_0)\big]^{-1}$ 
and Assumption \ref{assumption_inverse_existence} $(b)$.
Applying Corollary \ref{sign_eigenfunction} to the above ODE gives 
$w_n(t;\psi)\ge 0,$ i.e. $u_n(t;\psi)\ge b_n\ge 0$ on $[0,T].$ 
Hence, 
$$
K\psi(t)=\frac{g(t)}{\sum\limits_{n=1}^\infty u_n(t;\psi)
\frac{\partial \phi_n}{\n}(x_0)} \le 
\frac{g(t)}{\sum\limits_{n=1}^\infty b_n\frac{\partial \phi_n}{\n}(x_0)}
=g(t)\big[\frac{\partial u_0}{\n}(x_0)\big]^{-1}
$$
and this proof is complete.
\end{proof}

\par The existence conclusion is derived from Lemmas \ref{uniqueness_series} 
and \ref{itself_existence}.
\begin{theorem}[Existence]\label{existence}
\par Suppose Assumptions \ref{assumption_inverse} and 
\ref{assumption_inverse_existence} be valid, then there exists a 
fixed point of $K$ in $\DK'.$
\end{theorem}
\begin{proof}
 \par Lemma \ref{uniqueness_series} yields the sequence 
 $\{\overline{a}_n:\NN\}$ is increasing, while Lemma \ref{itself_existence} 
 gives $\{\overline{a}_n:\NN\}\subset \DK'.$ 
 Then $\{\overline{a}_n:\NN\}$ is an increasing sequence with an
 upper bound $g(t)\big[\frac{\partial u_0}{\n}(x_0)\big]^{-1},$ 
 which implies the convergence of $\{\overline{a}_n:\NN\}.$ 
 Denote the limit by $\overline{a},$ clearly $\overline{a}$ 
 is a fixed point of $K.$ Also, the closedness of $\DK'$ 
 yields that $\overline{a}\in \DK'.$  
 Therefore, $\overline{a}$ is a fixed point of $K$ in $\DK',$ which 
 confirms the existence.
\end{proof}

\subsection{Main theorem for the inverse problem and reconstruction algorithm}

\par Lemma \ref{uniqueness_series}, Theorems \ref{uniqueness} and 
\ref{existence} allow us to deduce the main theorem for this 
inverse problem.

\begin{theorem}[Main theorem for the inverse problem]\label{main_inverse}
 Suppose Assumption \ref{assumption_inverse} holds. 
 \begin{itemize}
  \item [(a)] If there exists a fixed point of K in $\DK,$ then it is 
  unique and coincides with the limit of $\{\overline{a}_n:\NN\};$ 
  \item [(b)] If Assumption \ref{assumption_inverse_existence} is also valid, 
  then there exists a unique fixed point of K in $\DK',$ which is the 
  limit of $\{\overline{a}_n:\NN\}.$
 \end{itemize}
\end{theorem}

\par The following reconstruction algorithm for $a(t)$ is based on 
Theorem \ref{main_inverse}.
\begin{table}[h!]
\caption{Numerical Algorithm}\label{algorithm}
\begin{tabular*}{16cm}{l}
     \hline
     Iteration algorithm to recover 
     the coefficient $a(t)$  \\
     \hline
1: Set up the right-hand side function $F(x,t)$ and the initial 
condition $u_0(x)$,\\
then measure the output flux data $g(t).$ 
$F$, $u_0$ and $g$ should satisfy Assumption \ref{assumption_inverse};\\
2: Set the initial guess as 
$\overline{a}_0(t)=g(t)\Big[\frac{\partial u_0}{\n}(x_0)
+ \I[\frac{\partial F}{\n}(x_0,t)]\Big]^{-1};$\\
3: {\bf for k = 1,...,N do}\\
4: Using the L1 time-stepping \cite{JinLazarovZhou:L1}
to compute $u(x,t;\overline{a}_{k-1})$,\\
which is the weak solution of FDE \eqref{fde} 
with coefficient function $\overline{a}_{k-1}$;\\
5: Update the coefficient $\overline{a}_{k-1}$ by 
$ \overline{a}_{k}=K\overline{a}_{k-1};$\\
6: Check stopping criterion $\|\overline{a}_{k}-
\overline{a}_{k-1}\|_{L^2[0,T]}\le \epsilon_0$
for some $\epsilon_0>0$;\\
7: {\bf end for}\\
8: {\bf output} the approximate coefficient function $\overline{a}_{N}$.\\
\hline
\end{tabular*}
\end{table}

\section{Numerical Results for inverse problem}

\subsection{L1 time-stepping of $\D$}
\par The fourth step of Algorithm \ref{algorithm} includes   
solving the direct problem of FDE \eqref{fde} numerically. 
To this end, we choose L1 time 
stepping \cite{JinLazarovZhou:L1,LinXu:2007} to discretize the term 
$\D u(x,t):$
 \begin{equation*}
   \begin{aligned}
     \D u(x,t_N) &= \frac{1}{\Gamma(1-\alpha)}\sum^{N-1}_{j=0}
     \int^{t_{j+1}}_{t_j} \frac{\partial u(x,s)}{\partial s} 
     (t_N-s)^{-\alpha}\, ds \\
     &\approx \frac{1}{\Gamma(1-\alpha)}\sum^{N-1}_{j=0} 
     \frac{u(x,t_{j+1})-u(x,t_j)}{\tau}\int_{t_j}^{t_{j+1}}
     (t_N-s)^{-\alpha}ds\\
     &=\sum_{j=0}^{N-1}b_j\frac{u(x,t_{N-j})-u(x,t_{N-j-1})}
     {\tau^\alpha}\\
     &=\tau^{-\alpha} [b_0u(x,t_N)-b_{N-1}u(x,t_0)
     +\sum_{j=1}^{N-1}(b_j-b_{j-1})u(x,t_{N-j})] ,
   \end{aligned}
 \end{equation*}
where 
\begin{equation*}
b_j=((j+1)^{1-\alpha}-j^{1-\alpha})/\Gamma(2-\alpha),\ j=0,1,\ldots,N-1.
\end{equation*}

\subsection{Numerical results for noise free data}
\par In this part, we set $\Omega=(0,1),\ x_0=0,\ T=1,\ \L u=u_{xx},$ 
pick $u_0(x)=-\sin{\pi x}, \ F(x,t)=-(t+1)\sin{\pi x}$ 
and consider the following two coefficients:
\begin{itemize}
 \item [(a1)] smooth coefficient: $a(t)=\sin{5\pi t}+1.3;$
 \item [(a2)] nonsmooth coefficient (``smile'' function): 
 \begin{equation*}
 \begin{split}
 a(t)&=[0.8\sin{3\pi t}+1.5]\chi_{[0,1/3]}
 +[-0.5\sin{(3\pi t-\pi)}+0.6]\chi_{(1/3,2/3)}\\
 &\quad+[0.8\sin{(3\pi t-2\pi)}+1.5]\chi_{[2/3,1]}.
 \end{split} 
 \end{equation*}
\end{itemize}

\par In experiment (a1), the exact coefficient we pick is a smooth 
function. Figure \ref{smooth_monotone} shows the initial guess and the 
first three iterations, while Figure \ref{smooth_unique_9} presents 
the exact and approximate coefficients. From these two figures, we  
observe that $\{\overline{a}_n:\NN\}$ converges to $a(t)$ monotonically, 
which illustrates Theorems \ref{monotonicity} and \ref{main_inverse}. 
Moreover, the $L^2$ error of the approximation in Figure 
\ref{smooth_unique_9} is 
$\|a-\overline{a}_N\|_{L^2[0,T]}=1.04\times 10^{-6},$
which implies us the $L^2$ error of this approximation may be bounded 
by the stopping criterion number $\epsilon_0.$ This guess is confirmed 
by Figure \ref{error_epsilon} and can be expressed as 
$$\|a-\overline{a}_N\|_{L^2[0,T]}=O(\epsilon_0).$$ 
Several attempts of experiment (a1) for different $\alpha\in(0,1)$ 
are taken to find the dependence of the convergence rate of Algorithm 
\ref{algorithm} on the fractional order $\alpha,$ which is shown in Figure 
\ref{alpha_N}. This figure shows the amounts of iterations 
required, i.e. $N,$ corresponding to different $\alpha,$ which imply 
that restricted $\alpha\in(0,1),$ the larger $\alpha$ is, the faster the 
convergence rate of Algorithm \ref{algorithm} is. This phenomenon 
is explained in \cite{jin2012inverse} by a property of the Mittag-Leffler 
function; for $\alpha\in(0,1),$ the larger $\alpha$ is, 
the faster the decay rate of $E_{\alpha,1}(-z)$ is as $z\to \infty.$

\begin{figure}[th!]
\center
\subfigure[$\alpha=0.3$]{
\includegraphics[trim = .5cm .15cm .5cm .3cm, clip=true,height=4.5cm,width=6cm]
{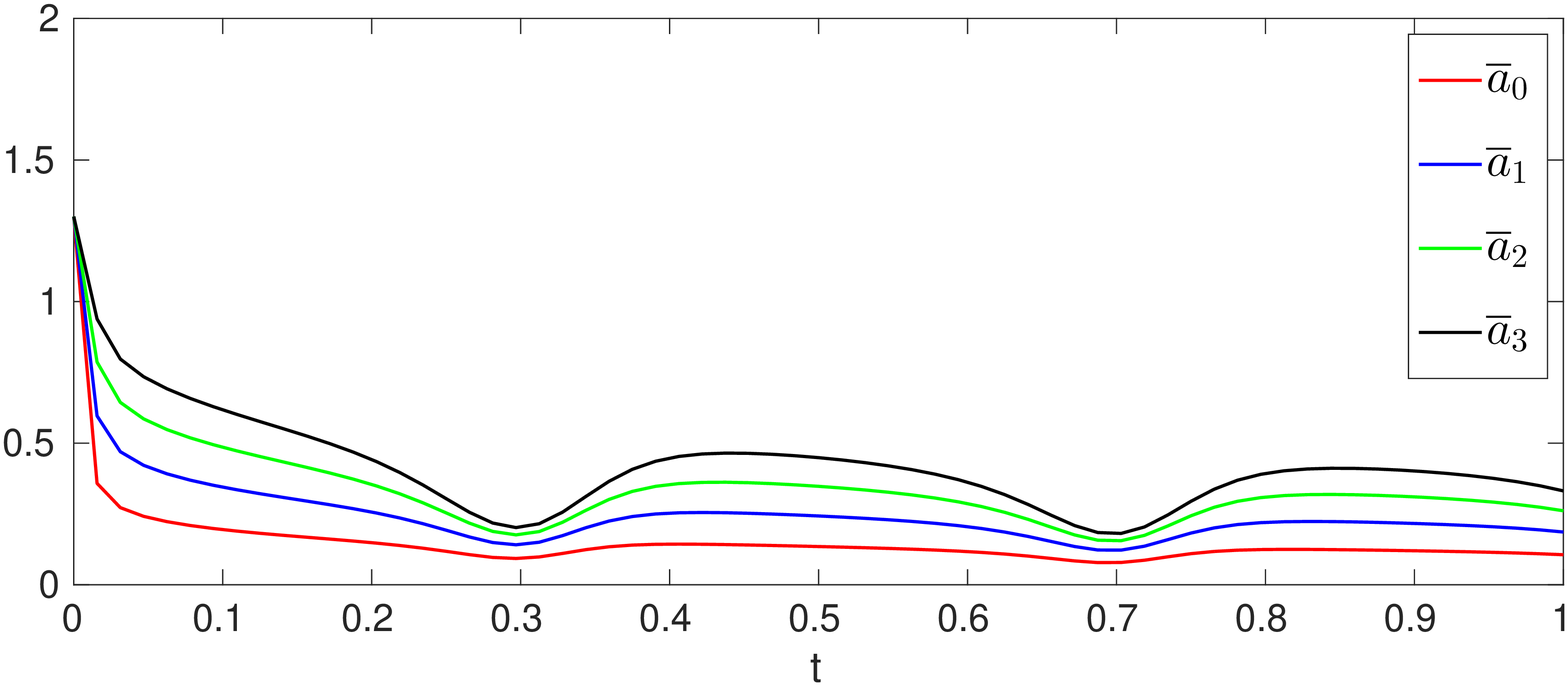}
}
\subfigure[$\alpha=0.5$]{
\includegraphics[trim = .5cm .15cm .5cm .3cm, clip=true,height=4.5cm,width=6cm]
{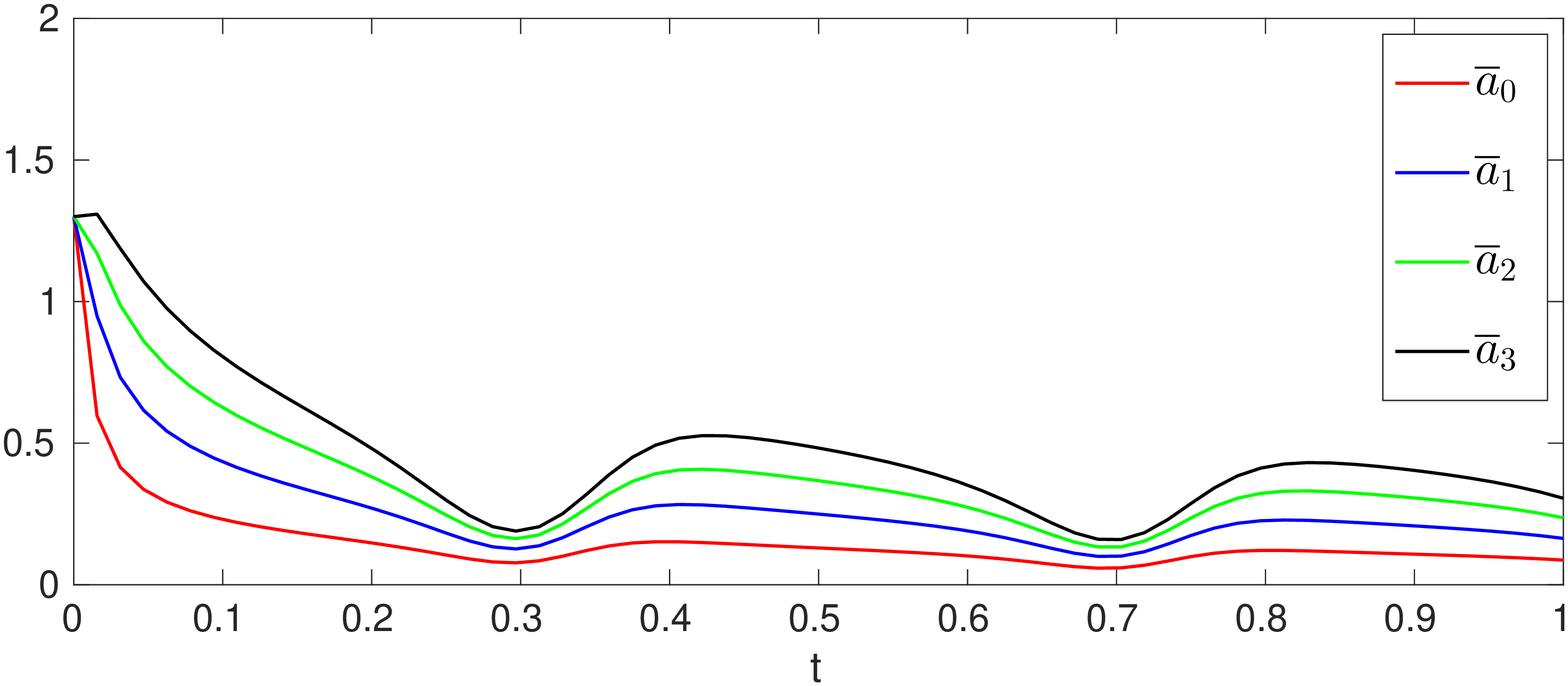}
}\\
\subfigure[$\alpha=0.7$]{
\includegraphics[trim = .5cm .15cm .5cm .3cm, clip=true,height=4.5cm,width=6cm]
{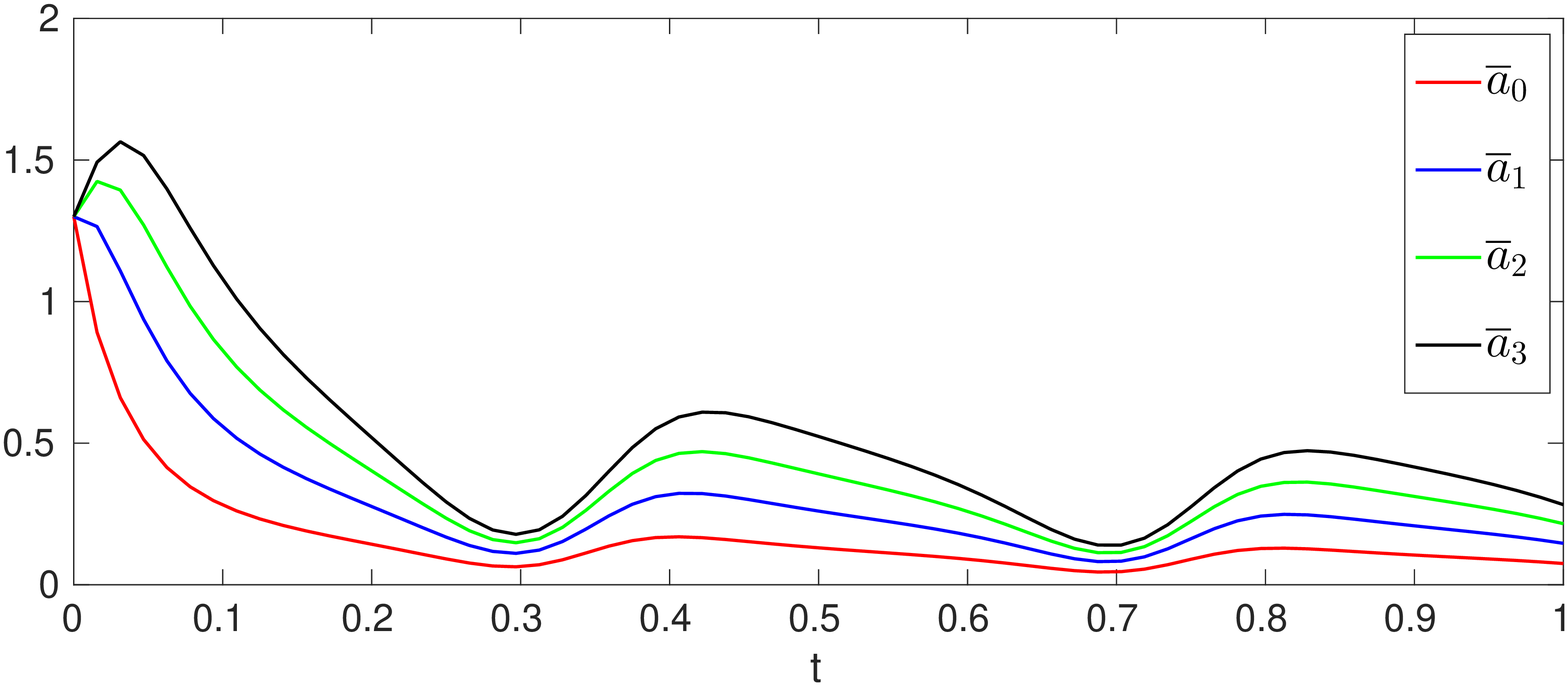}
}
\subfigure[$\alpha=0.9$]{
\includegraphics[trim = .5cm .15cm .5cm .3cm, clip=true,height=4.5cm,width=6cm]
{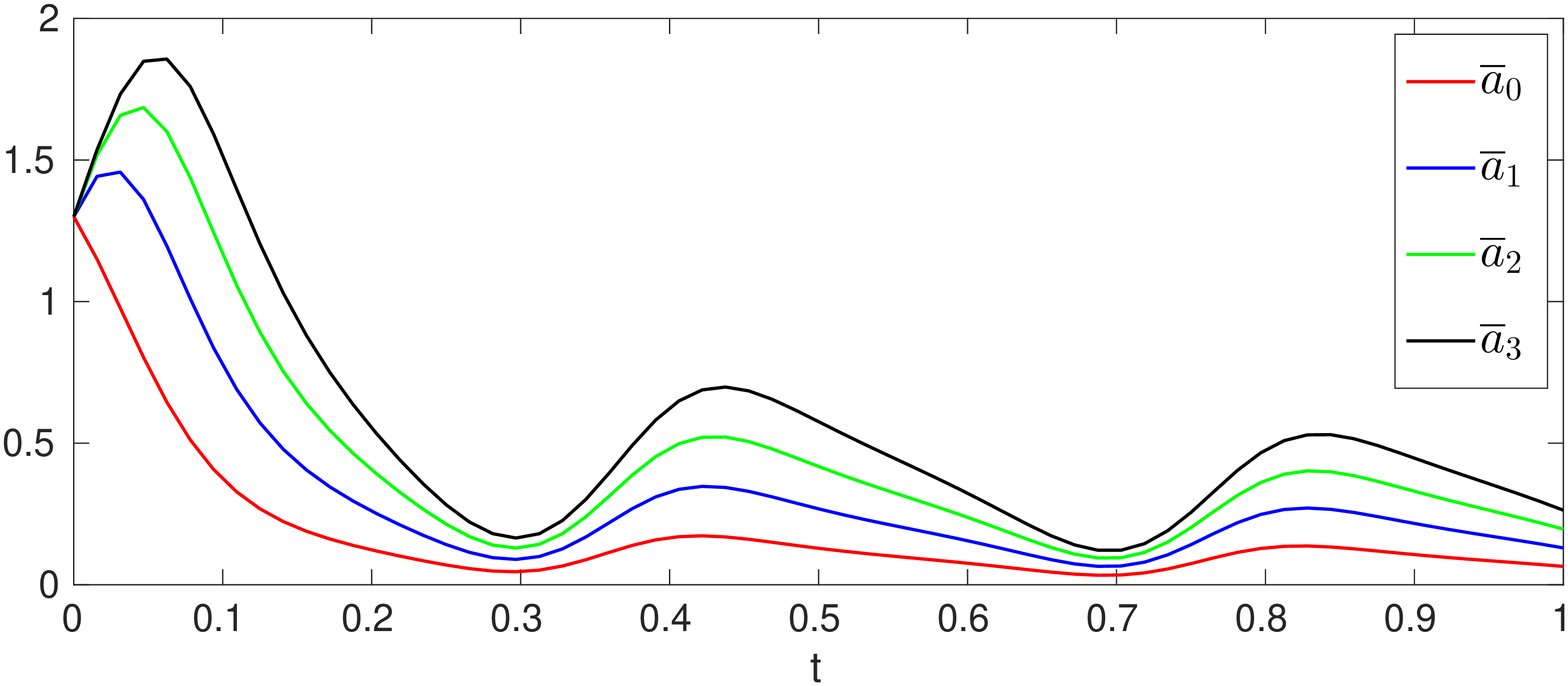}
}
\caption{Experiment (a1): the initial guess and first three iterations }
\label{smooth_monotone}
\end{figure}

\begin{figure}[h!]
\center
  \begin{tabular}{cc}
  \includegraphics[trim = .5cm .15cm .5cm .3cm, clip=true,height=5cm,width=8cm]{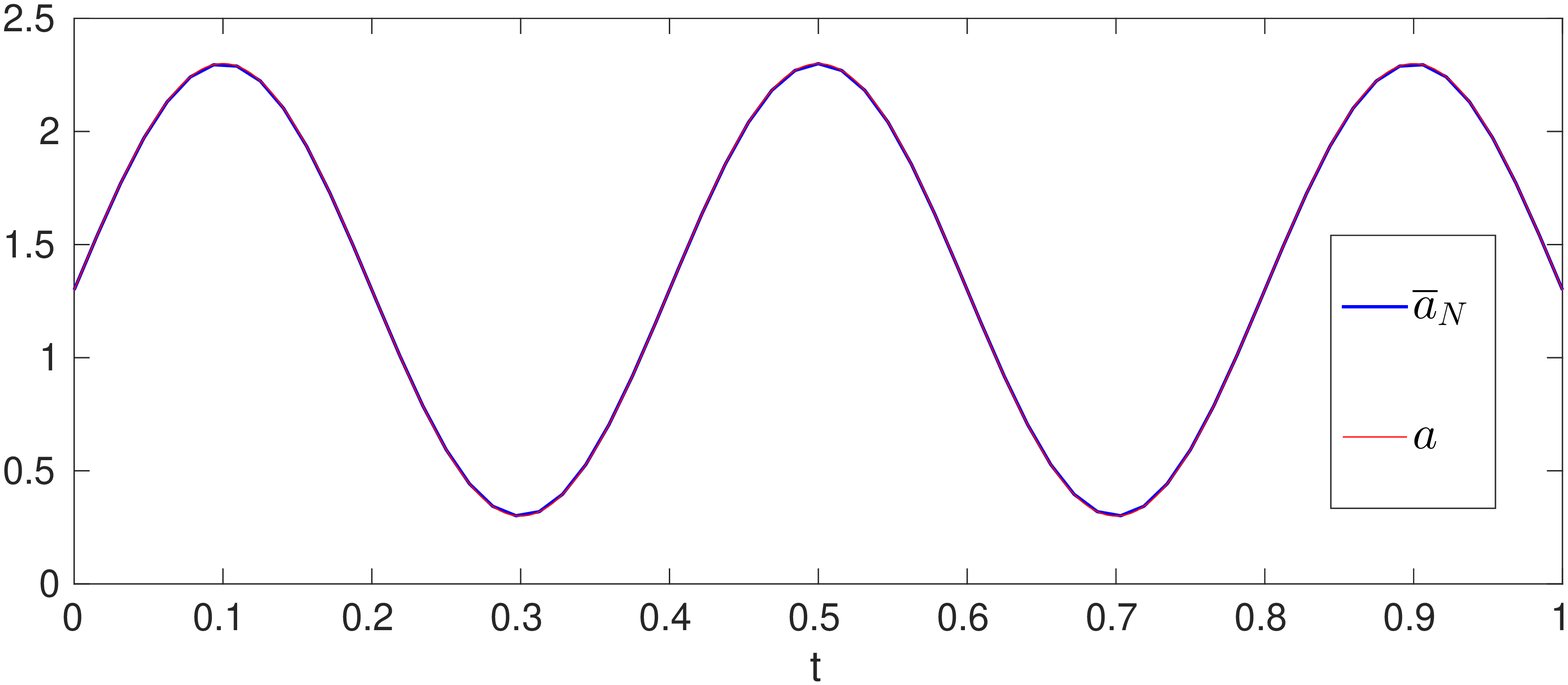}
  \end{tabular}
  \vspace{-.2cm}
  \caption{Experiment (a1): the exact and approximate coefficients for $\alpha=0.9$ and 
  $\epsilon_0=10^{-6}$}
  \label{smooth_unique_9}
\end{figure}

\begin{figure}[h!]
\center
  \begin{tabular}{cc}
  \includegraphics[trim = .5cm .15cm .5cm .3cm, clip=true,height=5cm,width=8cm]{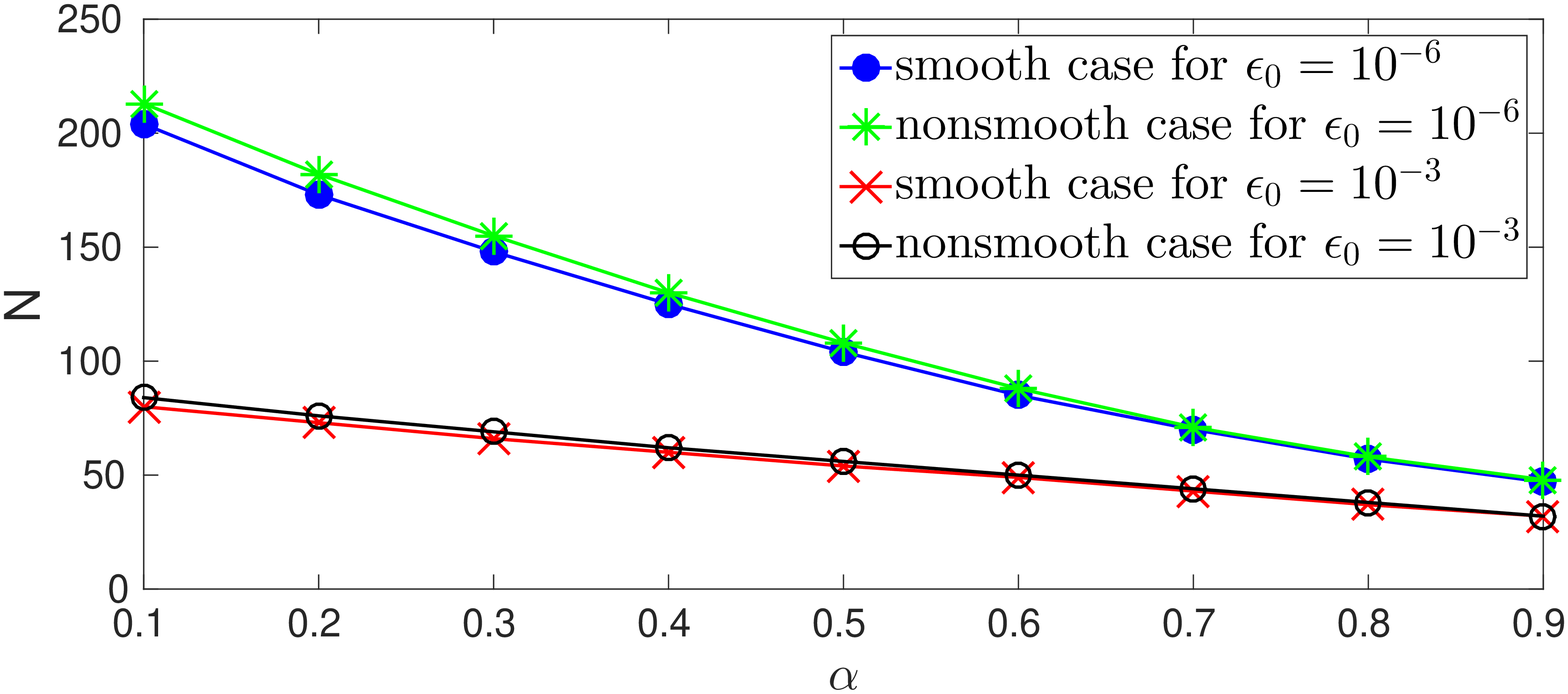}
  \end{tabular}
  \vspace{-.2cm}
  \caption{the amounts of iterations $N$ for different $\alpha$}
  \label{alpha_N}
\end{figure}

\begin{figure}[h!]
\center
  \begin{tabular}{cc}
  \includegraphics[trim = .5cm .15cm .5cm .3cm, clip=true,height=5cm,width=8cm]{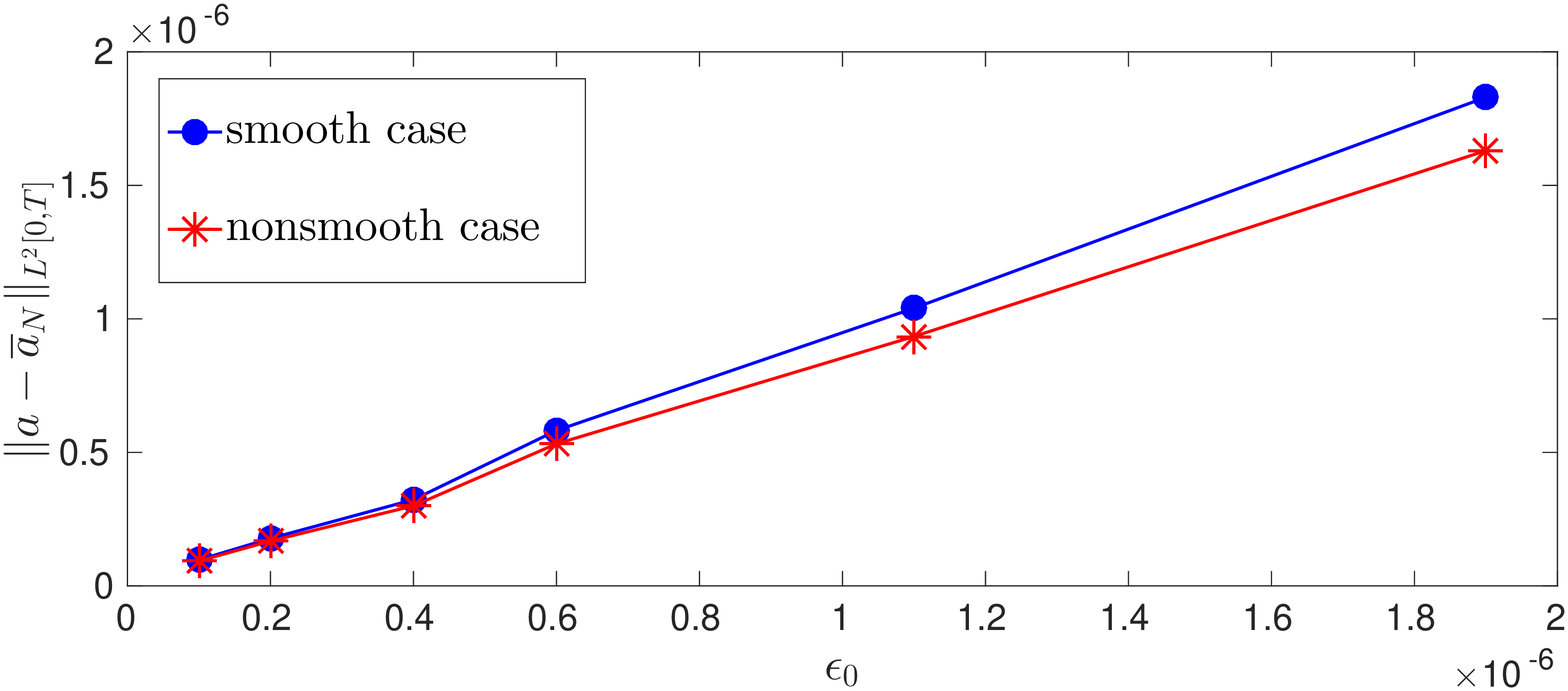}
  \end{tabular}
  \vspace{-.2cm}
  \caption{$\|a-\overline{a}_N\|_{L^2[0,T]}$ for different $\epsilon_0$ under $\alpha=0.9$}
  \label{error_epsilon}
\end{figure}

\par The definition of $\DK$ restricts the coefficient $a(t)$ in the space 
$C^+[0,T],$ however, the results of experiment (a2) indicate that Algorithm 
\ref{algorithm} still works for nonsmooth $a(t),$ which means the numerical 
restriction on $a(t)$ can possibly be extended from $a(t)\in C^{+}[0,T]$ to 
$a(t)\in L^\infty [0,T].$  
For discontinuous $a(t),$ Figures \ref{nonsmooth_monotone} 
and \ref{nonsmooth_unique_9} explain that 
Theorems \ref{monotonicity} and \ref{main_inverse} still hold, 
while Figures \ref{alpha_N} and \ref{error_epsilon}   
illustrate the similar conclusions as  
the larger $\alpha$ is, the faster the convergence rate of Algorithm 
\ref{algorithm} is, and 
$$\|a-\overline{a}_N\|_{L^2[0,T]}=O(\epsilon_0).$$

\begin{figure}[th!]
\center
\subfigure[$\alpha=0.3$]{
\includegraphics[trim = .5cm .15cm .5cm .3cm, clip=true,height=4.5cm,width=6cm]
{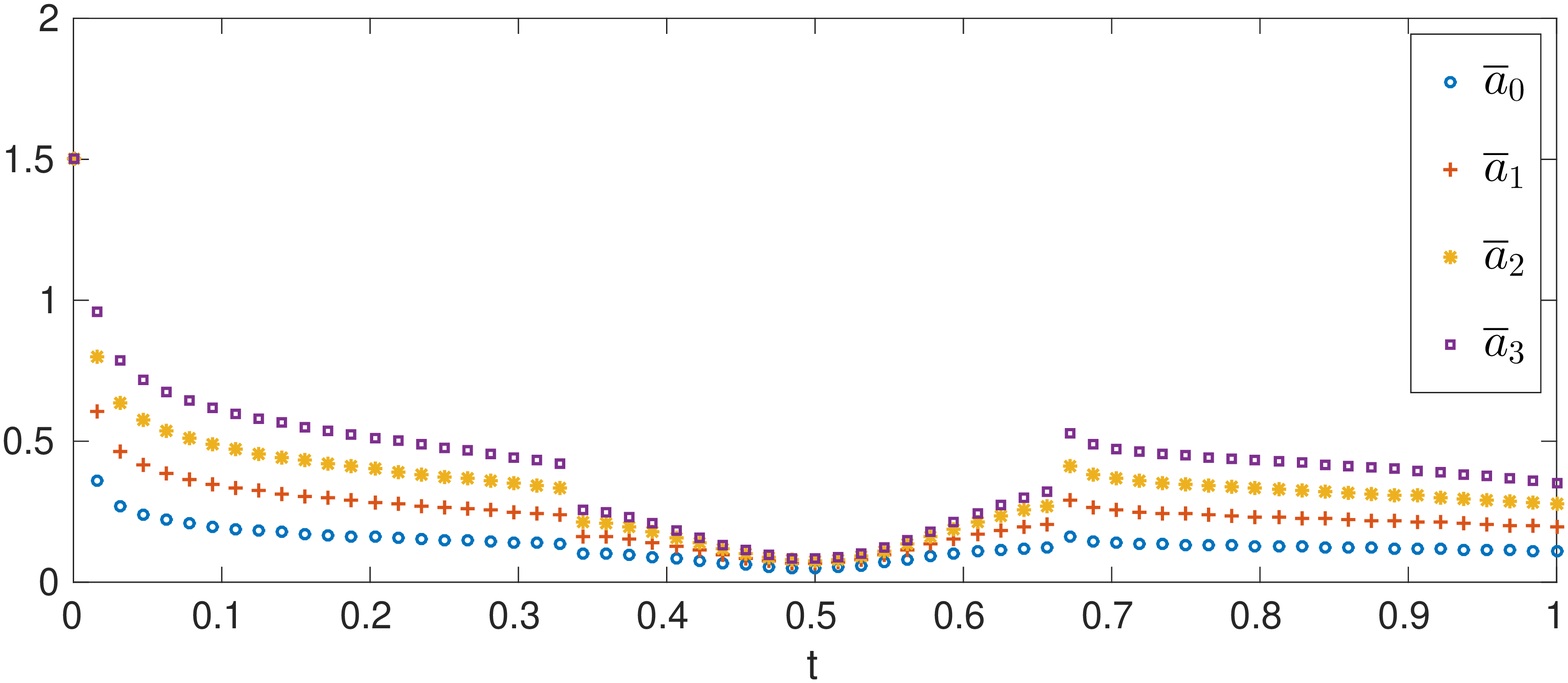}
}
\subfigure[$\alpha=0.5$]{
\includegraphics[trim = .5cm .15cm .5cm .3cm, clip=true,height=4.5cm,width=6cm]
{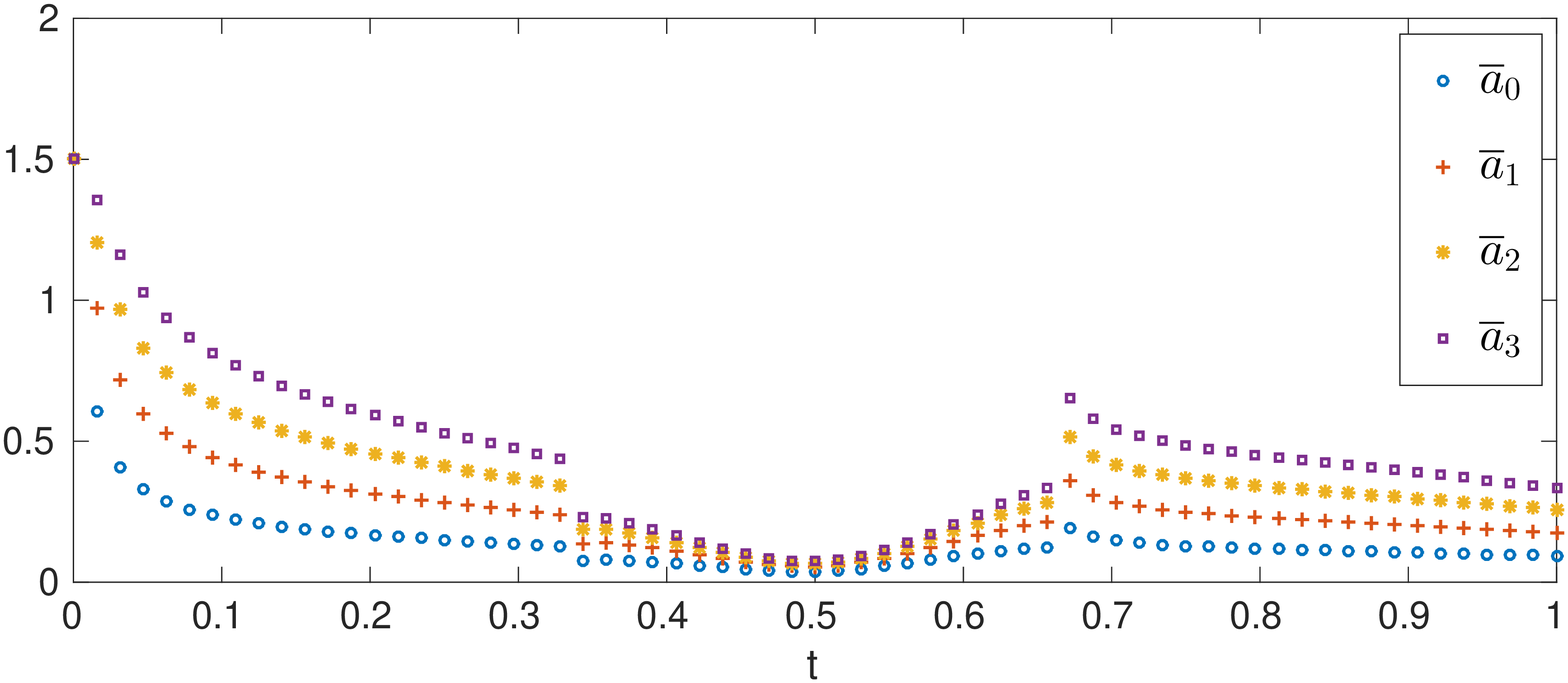}
}\\
\subfigure[$\alpha=0.7$]{
\includegraphics[trim = .5cm .15cm .5cm .3cm, clip=true,height=4.5cm,width=6cm]
{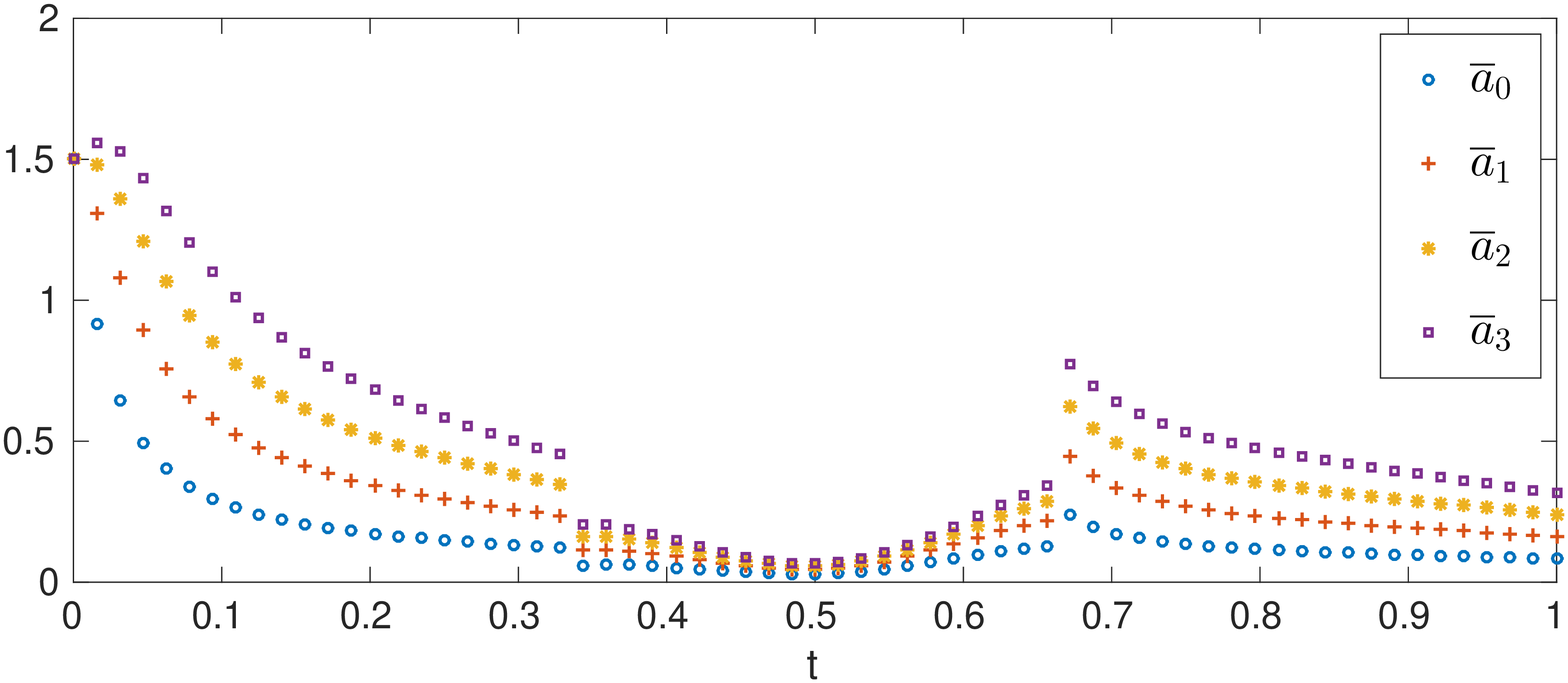}
}
\subfigure[$\alpha=0.9$]{
\includegraphics[trim = .5cm .15cm .5cm .3cm, clip=true,height=4.5cm,width=6cm]
{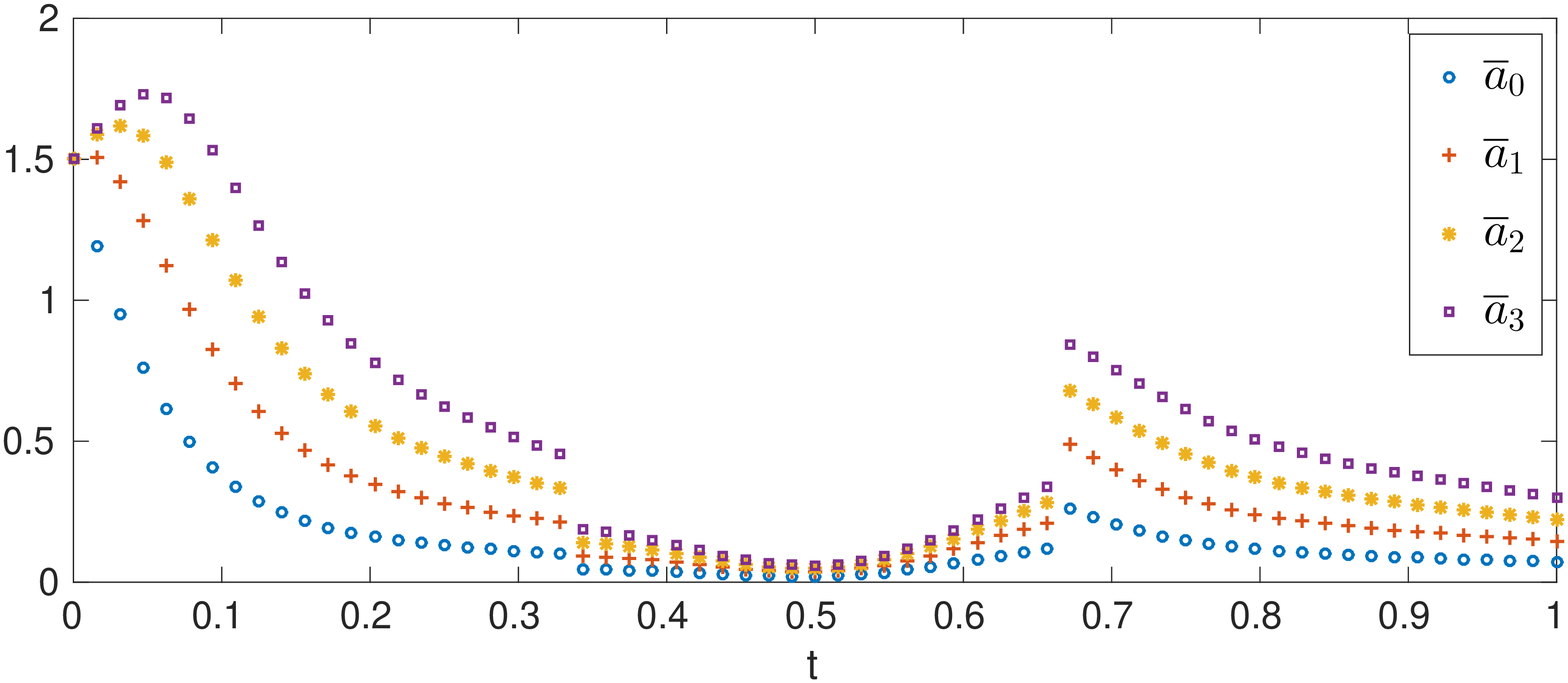}
}
\caption{Experiment (a2): the initial guess and first three iterations }
\label{nonsmooth_monotone}
\end{figure}

\begin{figure}[h!]
\center
  \begin{tabular}{cc}
  \includegraphics[trim = .5cm .15cm .5cm .3cm, clip=true,height=5cm,width=8cm]{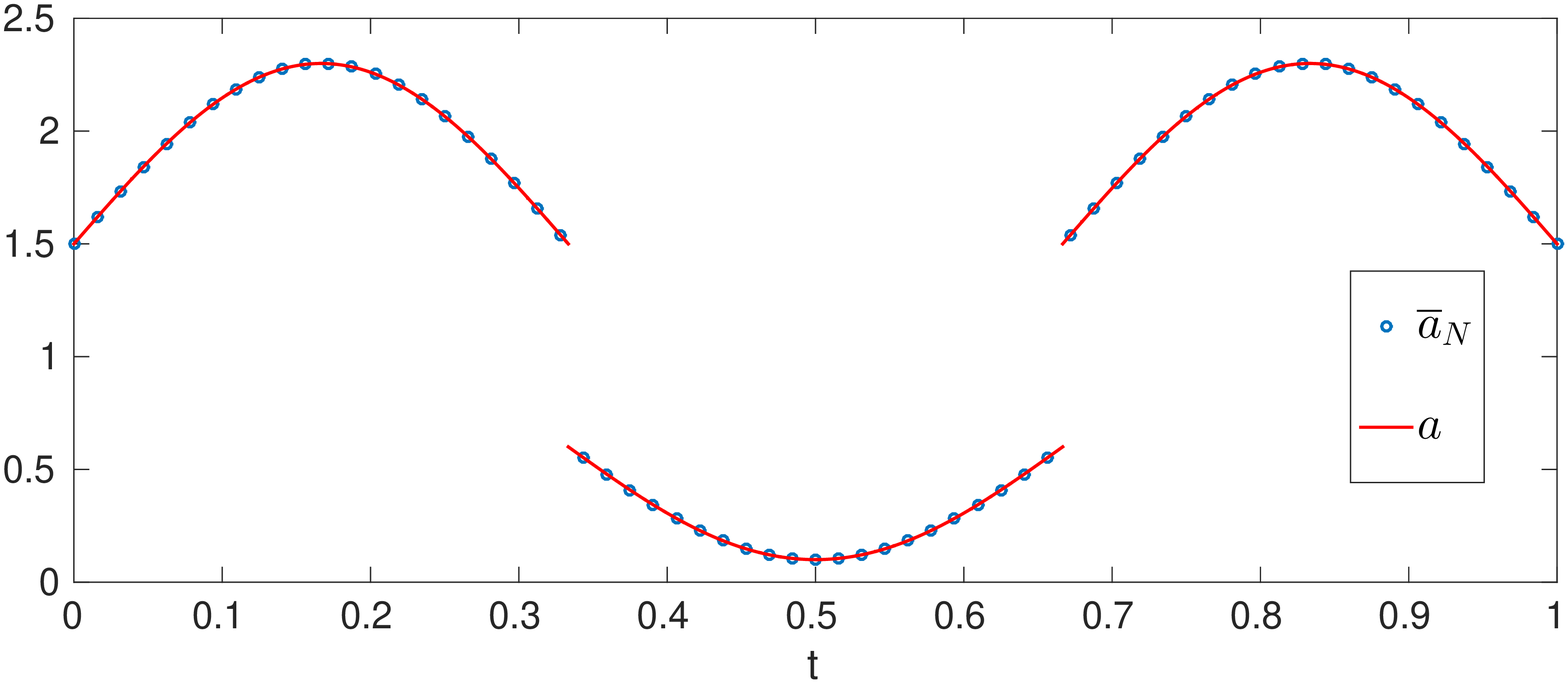}
  \end{tabular}
  \vspace{-.2cm}
  \caption{Experiment (a2): the exact and approximate coefficients for $\alpha=0.9$ and 
  $\epsilon_0=10^{-6}$}
  \label{nonsmooth_unique_9}
\end{figure}

\subsection{Numerical results for noisy data}
\par In this subsection, we will consider data polluted by noise.   
Set $g$ be the exact data and denote the noisy data by $g_\delta$ 
with relative noise level $\delta,$ i.e. 
$\|(g-g_\delta)/g\|_{L^\infty[0,T]}\le \delta.$ 
Then the perturbed operator $K_\delta$ is
$$K_\delta \psi(t)
=\frac{g_\delta(t)}{\sum\limits_{n=1}^\infty u_n(t;\psi)
\frac{\partial \phi_n}{\n}(x_0)}$$ 
with domain
$$\DKd:=\{\psi\in C^{+}[0,T]:g_\delta(t)\Big[\frac{\partial u_0}{\n}(x_0)
+ \I[\frac{\partial F}{\n}(x_0,t)]\Big]^{-1}\le \psi(t),\ t\in[0,T] \}.$$
Also, the sequence $\{\overline{a}_{\delta,n}:n\in \mathbb{N}\}$ can be 
obtained from the iteration 
$$\overline{a}_{\delta,0}=g_\delta\Big[\frac{\partial u_0}{\n}(x_0)
+ \I[\frac{\partial F}{\n}(x_0,t)]\Big]^{-1},\  
\overline{a}_{\delta,n+1}=K_\delta\overline{a}_{\delta,n},\ n\in \mathbb{N}.$$
Since $\delta$ is a small positive number and $g$ is a strictly 
positive function, we can assume $g_\delta$ is still positive, which 
means Theorem \ref{main_inverse} still holds for $K_\delta.$ 
Hence, if there exists a fixed point $a_\delta\in \DKd$, 
the sequence $\{\overline{a}_{\delta,n}:n\in \mathbb{N}\}$ will converge to 
$a_\delta$ monotonically and we denote the limit by $\overline{a}_\delta.$
Algorithm \ref{algorithm} is still able to be used to recover 
$\overline{a}_\delta$ after a slightly modification$-$replacing 
$g$ and $K$ by $g_\delta$ and $K_\delta,$ respectively.

\par We take the experiments (a1) and (a2) with  
noise level $\delta>0.$ Figures \ref{smooth_unique_9_noise_3} and 
\ref{nonsmooth_unique_9_noise_3} present 
the exact and approximate coefficients under $\delta=3\%$ for 
experiments (a1) and (a2) respectively. From figures  
\ref{smooth_unique_9_noise_3} and \ref{nonsmooth_unique_9_noise_3}, 
we observe that the smaller $|a(t)|$ is, the better the approximation is. 
This can be explained by $\delta$ means the relatively noise level, i.e. 
we pick $g_\delta=(1+\zeta\delta)g$ in the codes, where 
$\zeta$ follows a uniform distribution on $[-1,1].$
Figure \ref{error_delta} illustrates that  
$$\|a-\overline{a}_{\delta,N}\|_{L^2[0,T]}/\|a\|_{L^2[0,T]}=O(\delta),$$ 
showing the domination of the noise level $\delta$ in 
relatively $L^2$ error with the reason that $\epsilon_0 \ll \delta.$

\begin{figure}[h!]
\center
  \begin{tabular}{cc}
  \includegraphics[trim = .5cm .15cm .5cm .3cm, clip=true,height=5cm,width=8cm]{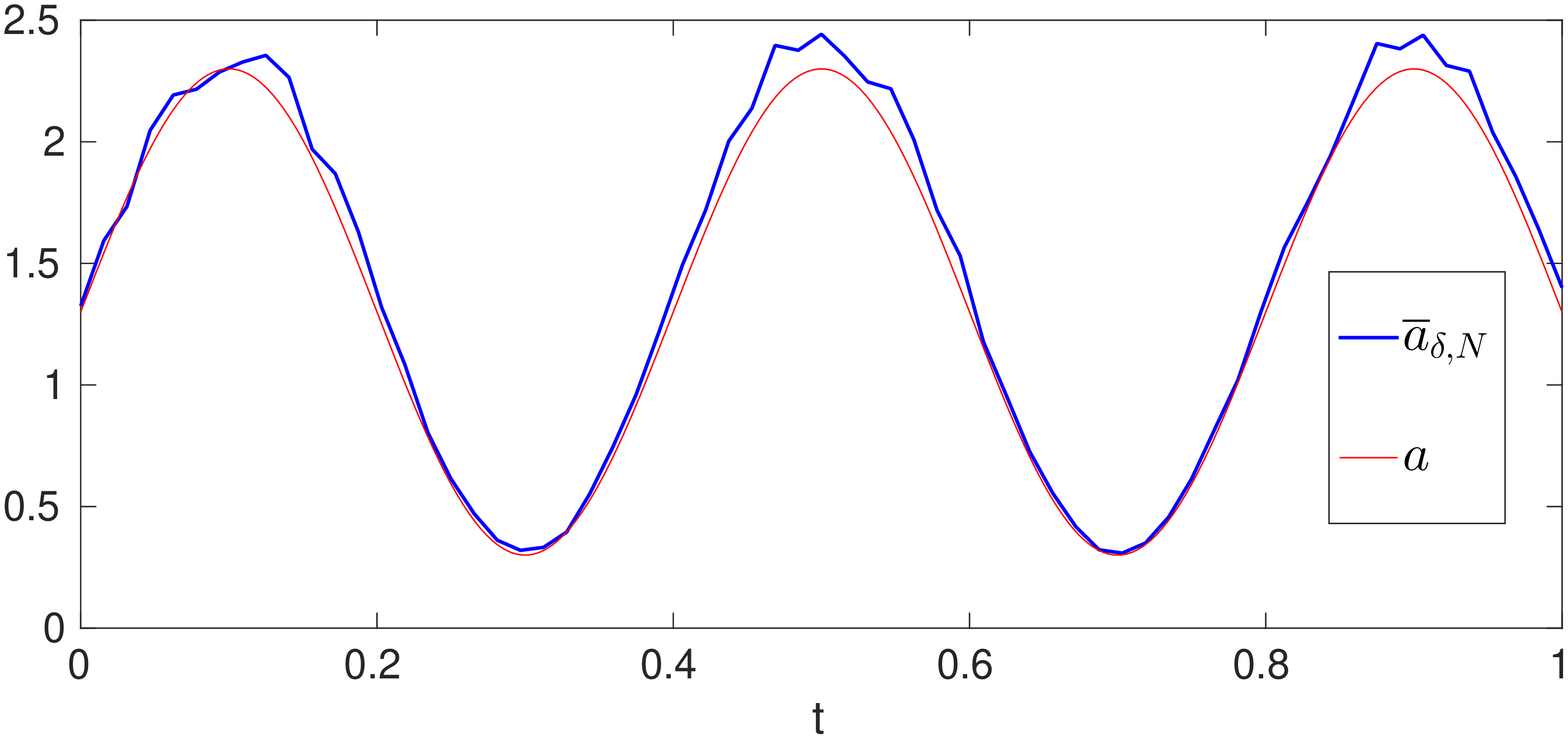}
  \end{tabular}
  \vspace{-.2cm}
  \caption{Experiment (a1): the exact and approximate coefficients with $\alpha=0.9,$ 
  $\epsilon_0=10^{-6}$ and $\delta=3\%$}
  \label{smooth_unique_9_noise_3}
\end{figure}

\begin{figure}[h!]
\center
  \begin{tabular}{cc}
  \includegraphics[trim = .5cm .15cm .5cm .3cm, clip=true,height=5cm,width=8cm]{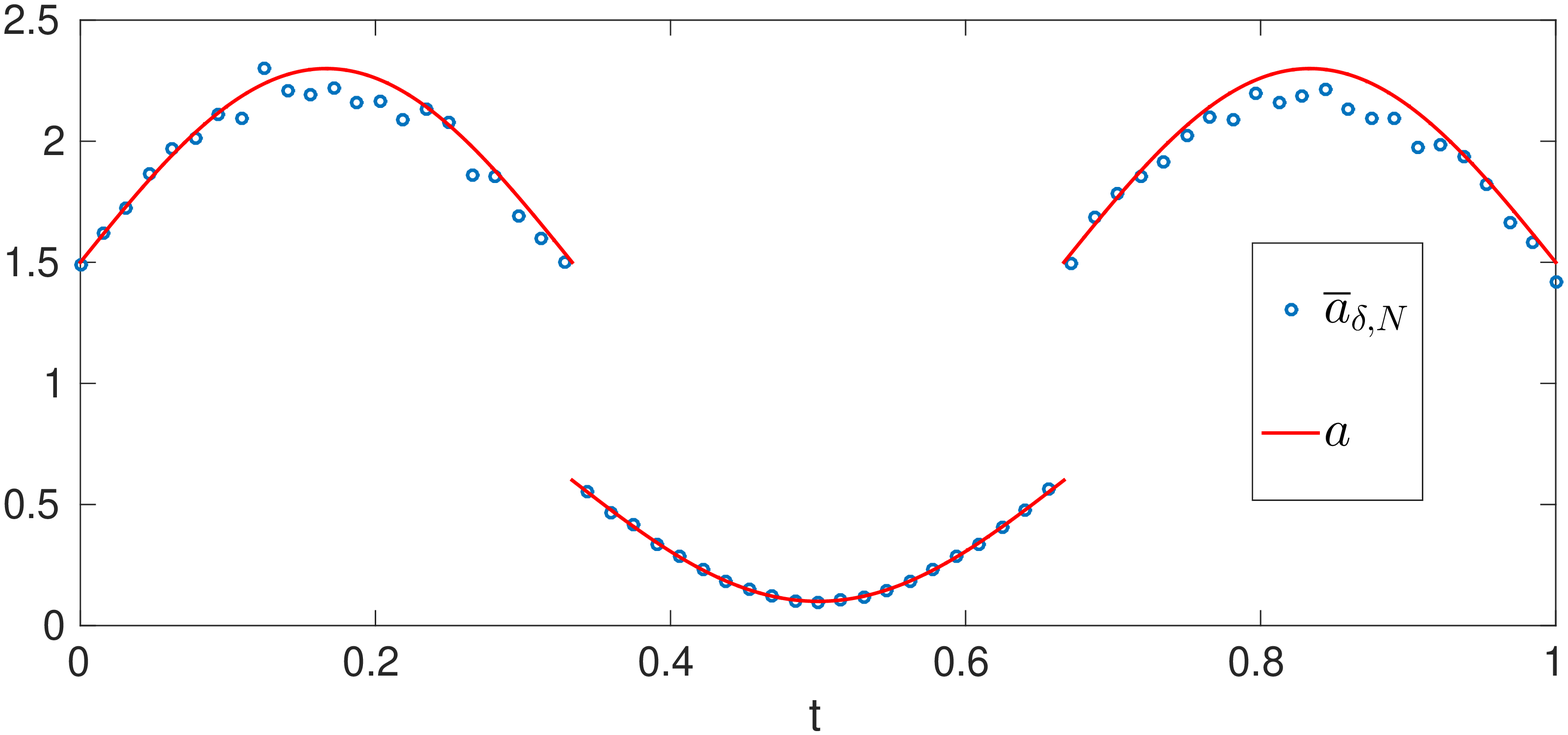}
  \end{tabular}
  \vspace{-.2cm}
  \caption{Experiment (a2): the exact and approximate coefficients with $\alpha=0.9,$ 
  $\epsilon_0=10^{-6}$ and $\delta=3\%$}
  \label{nonsmooth_unique_9_noise_3}
\end{figure}

\begin{figure}[h!]
\center
  \begin{tabular}{cc}
  \includegraphics[trim = .5cm .15cm .5cm .3cm, clip=true,height=5cm,width=8cm]{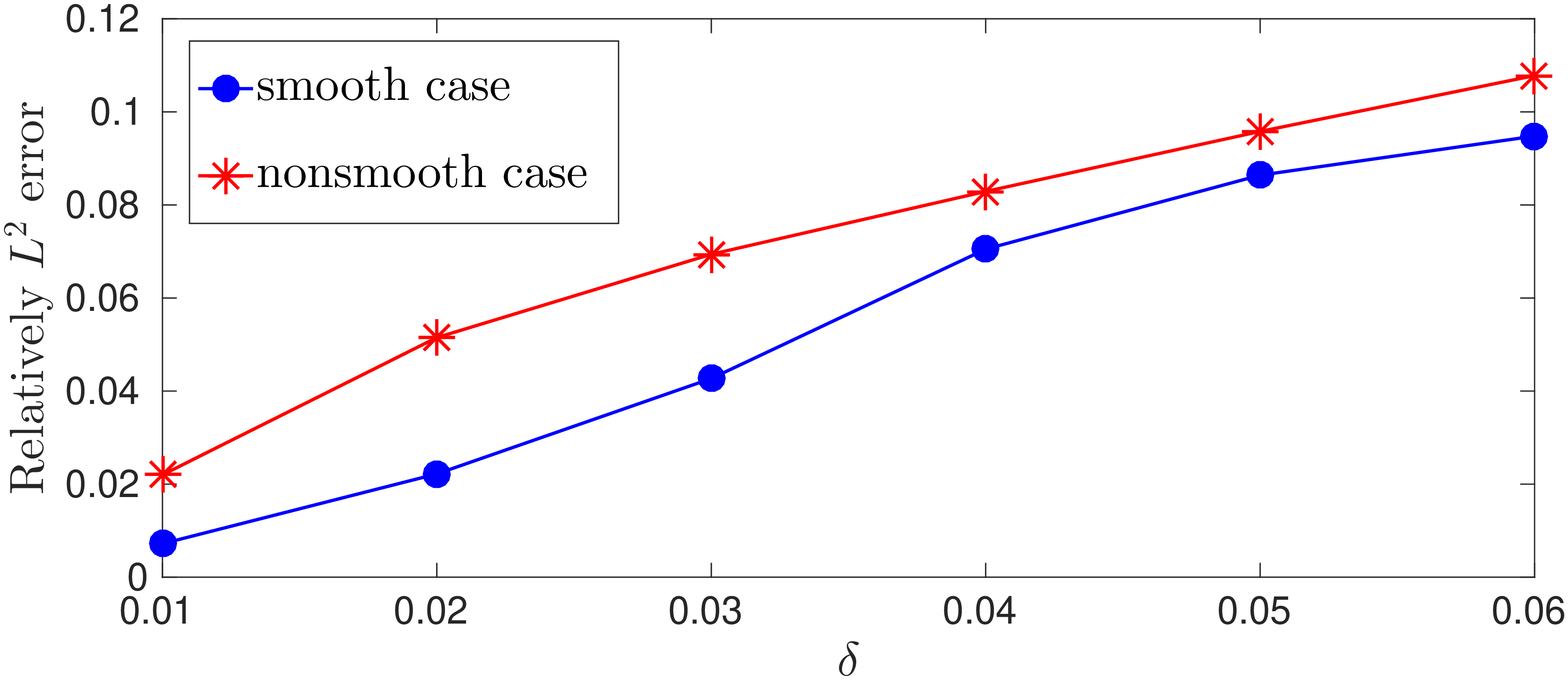}
  \end{tabular}
  \vspace{-.2cm}
  \caption{$\|a-\overline{a}_{\delta,N}\|_{L^2[0,T]}/\|a\|_{L^2[0,T]}$ 
  for different $\delta$ under $\alpha=0.9$ and $\epsilon_0=10^{-6}$}
  \label{error_delta}
\end{figure}

\subsection{Numerical results in two dimensional case}
In this part, the numerical experiments on a two dimensional domain will be considered. We set $\alpha=0.9,\ \epsilon_0=10^{-6},\ \Omega=(0,1)^2,\ x_0=(0,1/2),\ T=1,\ \L u=\triangle u,$ choose $u_0(x,y)=-\sin{[\pi xy(1-x)(1-y)]},$ $F(x,y)=-(t+1)\cdot\sin{[\pi xy(1-x)(1-y)]},$ and consider experiments (a1) and (a2). 
Figures \ref{smooth_monotone_2d} and \ref{nonsmooth_monotone_2d} confirm the theoretical conclusions in section 4.
 
\begin{figure}[th!]
	\center
	\subfigure[Initial guess and  first three iterations]{
		\includegraphics[trim = .5cm .15cm .5cm .3cm, clip=true,height=4.5cm,width=6cm]
		{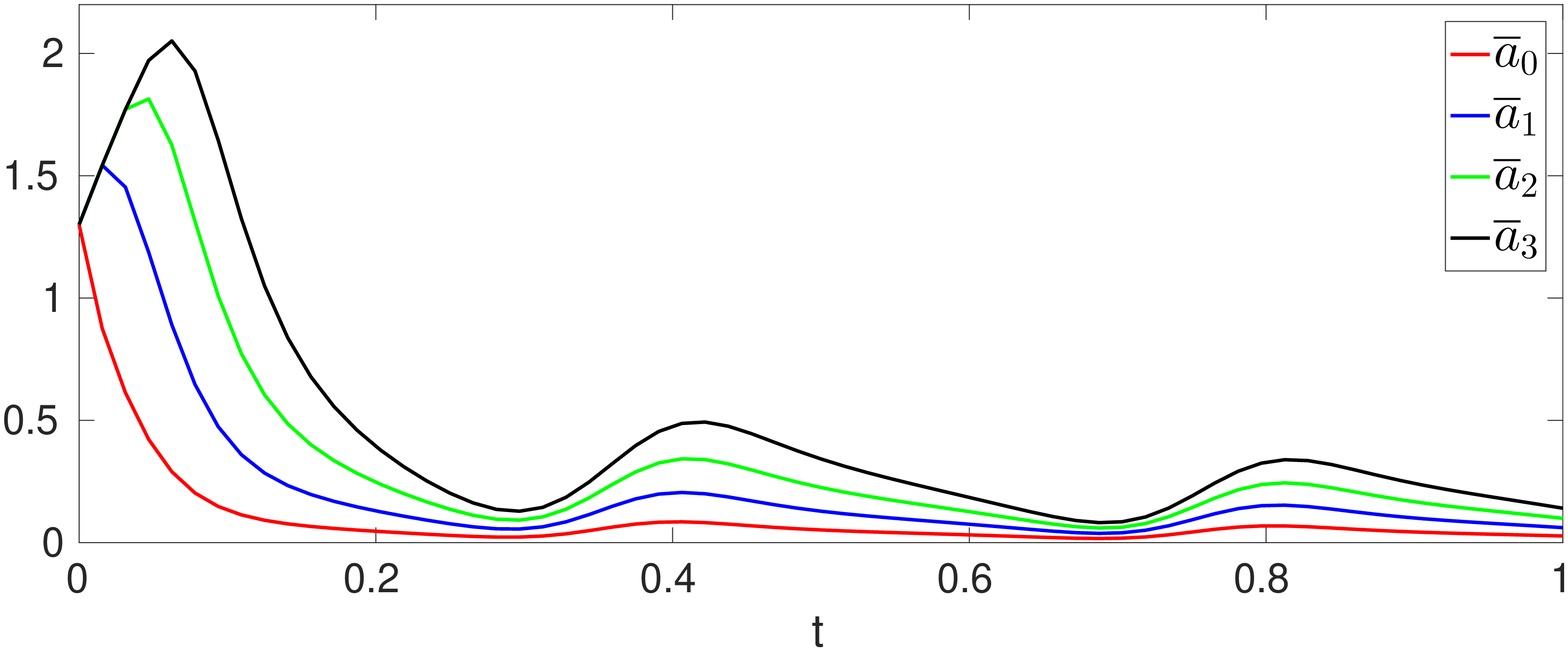}
	}
	\subfigure[Exact and approximate coefficients]{
		\includegraphics[trim = .5cm .15cm .5cm .3cm, clip=true,height=4.5cm,width=6cm]
		{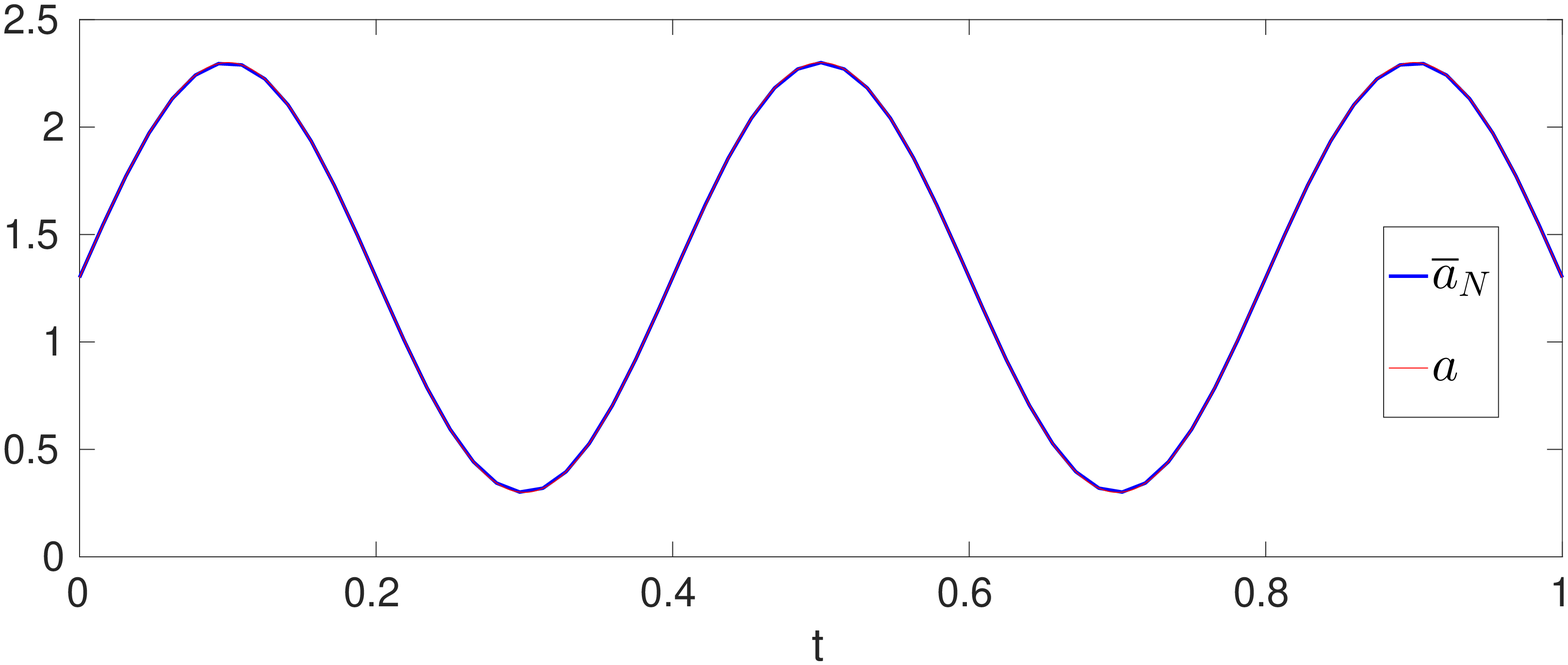}
	}
	\caption{Experiment (a1) in two dimensional case}
	\label{smooth_monotone_2d}
\end{figure}
\begin{figure}[th!]
	\center
	\subfigure[Initial guess and  first three iterations]{
		\includegraphics[trim = .5cm .15cm .5cm .3cm, clip=true,height=4.5cm,width=6cm]
		{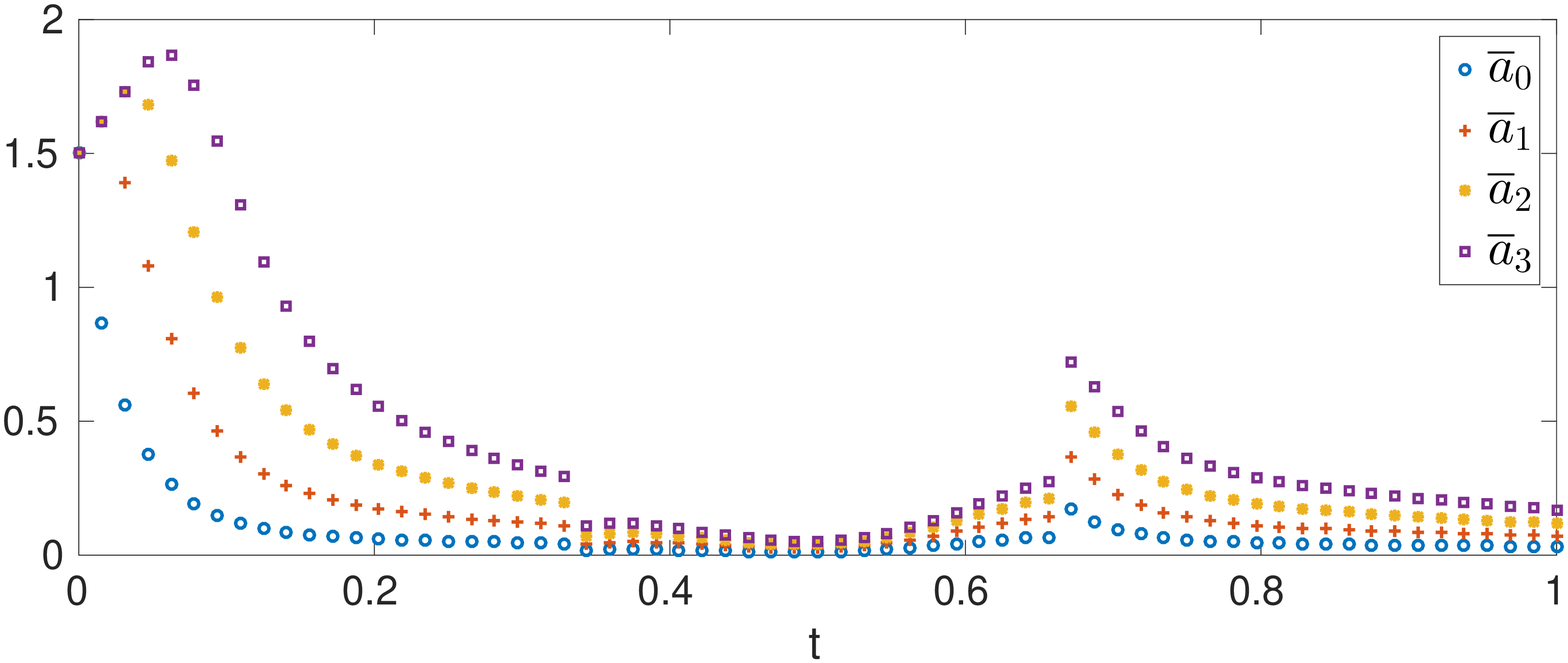}
	}
	\subfigure[Exact and approximate coefficients]{
		\includegraphics[trim = .5cm .15cm .5cm .3cm, clip=true,height=4.5cm,width=6cm]
		{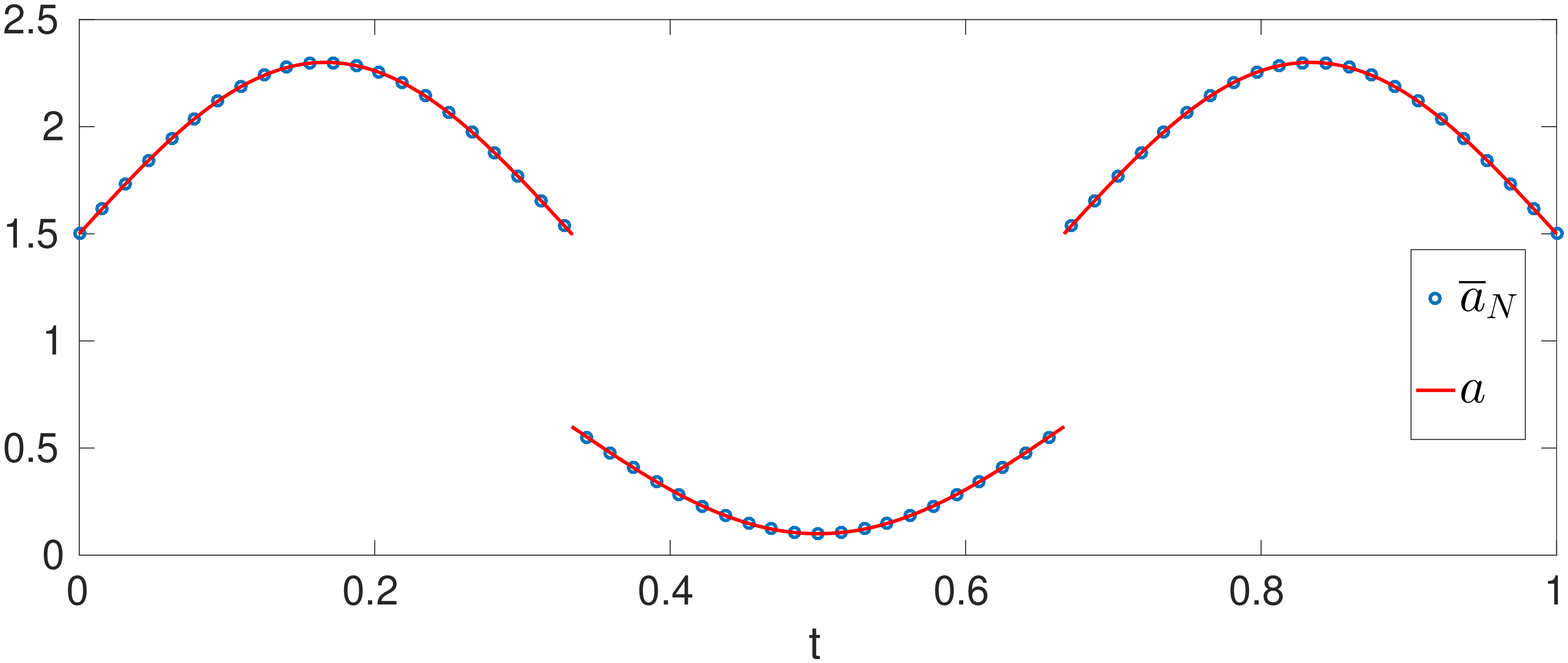}
	}
	\caption{Experiment (a2) in two dimensional case}
	\label{nonsmooth_monotone_2d}
\end{figure}

\section*{Acknowledgment}
The author is indebted to William Rundell for assistance in this work 
and acknowledges partial support from NSF-DMS 1620138. 

\bibliographystyle{abbrv}
\bibliography{frankjones_frac_multi}

\end{document}